\pdfoutput=1
\documentclass[a4paper,11pt]{amsart}

\usepackage[margin=1in,includehead,includefoot]{geometry}
\usepackage[utf8]{inputenc}
\usepackage[T1]{fontenc}
\usepackage{lmodern,microtype}
\usepackage{amsmath,amsthm,amssymb,colonequals}
\usepackage{enumitem}
\usepackage{caption}
\usepackage{graphicx}
\usepackage[unicode]{hyperref}
\usepackage{bookmark}

\hypersetup{colorlinks=true,linkcolor=black,citecolor=black,filecolor=black,urlcolor=black}

\makeatletter
\let\citationorig\citation
\def\citation#1{\citationorig{#1}\@for\@tempa:=#1\do{\@ifundefined{cit@\@tempa}{\global\@namedef{cit@\@tempa}{}}{}}}
\let\bibitemorig\bibitem
\def\bibitem#1{\@ifundefined{cit@#1}{\typeout{warning: Unused bibitem `#1'}}{}\bibitemorig{#1}}
\let\old@setaddresses\@setaddresses
\def\@setaddresses{\medskip{\parindent 0pt\let\scshape\relax\let\ttfamily\relax\old@setaddresses}}
\makeatother

\newtheorem{theorem}{Theorem}[section]
\newtheorem{lemma}[theorem]{Lemma}
\newtheorem{corollary}[theorem]{Corollary}
\newtheorem{proposition}[theorem]{Proposition}
\newtheorem{conjecture}[theorem]{Conjecture}
\newtheorem{observation}[theorem]{Observation}
\newtheorem{fact}{Fact}
\theoremstyle{remark}
\newtheorem{claim}{Claim}
\newtheorem*{claim*}{Claim}

\renewenvironment{enumerate}{\begin{enumorig}[label=\textup{(\arabic*)}, noitemsep, topsep=3pt plus 3pt, leftmargin=*]}{\end{enumorig}}

\renewenvironment{itemize}{\begin{itemorig}[label=\textbullet, noitemsep, topsep=3pt plus 3pt, leftmargin=1.1em]}{\end{itemorig}}

\DeclareMathOperator{\beer}{b} 
\DeclareMathOperator{\conv}{c}
\DeclareMathOperator{\smc}{smc}

\DeclareMathOperator{\Aff}{Aff}
\let\Box\undefined 
\DeclareMathOperator{\Box}{Box} 
\DeclareMathOperator{\Conv}{Conv} 
\DeclareMathOperator{\Seg}{Seg} 
\DeclareMathOperator{\Simp}{Simp} 
\DeclareMathOperator{\Skel}{Skel}
\DeclareMathOperator{\Vis}{Vis} 

\newcommand{\bd}{\partial} 
\newcommand{\inte}[1]{{#1}^\circ} 
\newcommand{\nei}[2]{\mathcal{N}_{#1}(#2)} 

\newcommand{\bbR}{\mathbb{R}} 
\newcommand{\bbN}{\mathbb{N}} 
\newcommand{\bbQ}{\mathbb{Q}}

\newcommand{\calC}{\mathcal{C}}
\newcommand{\calE}{\mathcal{E}}
\newcommand{\calP}{\mathcal{P}}
\newcommand{\calR}{\mathcal{R}}
\newcommand{\calU}{\mathcal{U}}

\newcommand{\fB}{\mathfrak{B}}
\newcommand{\fBv}{\fB_|}
\newcommand{\fBl}{\fB_{\lhd}}
\newcommand{\fBr}{\fB_{\rhd}}
\newcommand{\fBo}{\fB_{\bullet}}
\newcommand{\fS}{\mathfrak{S}}
\newcommand{\fT}{\mathfrak{T}}
\newcommand{\fV}{\mathfrak{V}}

\let\leq\leqslant
\let\geq\geqslant
\let\setminus\smallsetminus

\allowdisplaybreaks

\hypersetup{
  pdftitle={On the Beer index of convexity and its variants},
  pdfauthor={Martin Balko, Vít Jelínek, Pavel Valtr, Bartosz Walczak},
}

\title{On the Beer index of convexity and its variants}
\author{Martin Balko\and Vít Jelínek\and Pavel Valtr\and Bartosz Walczak}

\address[Martin Balko, Pavel Valtr]{Department of Applied Mathematics, Faculty of Mathematics and Physics, Charles University, Prague, Czech Republic}
\email{\href{mailto:balko@kam.mff.cuni.cz}{balko@kam.mff.cuni.cz}, \href{mailto:valtr@kam.mff.cuni.cz}{valtr@kam.mff.cuni.cz}}

\address[Vít Jelínek]{Computer Science Institute, Faculty of Mathematics and Physics, Charles University, Prague, Czech Republic}
\email{\href{mailto:jelinek@iuuk.mff.cuni.cz}{jelinek@iuuk.mff.cuni.cz}}

\address[Bartosz Walczak]{Department of Theoretical Computer Science, Faculty of Mathematics and Computer Science, Jagiellonian University, Kraków, Poland}
\email{\href{mailto:walczak@tcs.uj.edu.pl}{walczak@tcs.uj.edu.pl}}

\thanks{A journal version of this paper appeared in \href{http://doi.org/10.1007/s00454-016-9821-3}{\emph{Discrete Comput.\ Geom.}\ 57~(1), 179--214, 2017}.}
\thanks{A preliminary version of this paper appeared in: \href{http://doi.org/10.4230/LIPIcs.SOCG.2015.406}{Lars Arge and János Pach (eds.), \emph{31st International Symposium on Computational Geometry (SoCG 2015)}, vol.~34 of \emph{Leibniz International Proceedings in Informatics (LIPIcs)}, pp.~406--420, Leibniz-Zentrum für Informatik, Dagstuhl, 2015}.}
\thanks{The first three authors were supported by the grant GAČR 14-14179S.
The first author acknowledges the support of the Grant Agency of the Charles University, GAUK 690214 and the grant SVV-2015-260223.
The last author was supported by the Ministry of Science and Higher Education of Poland \emph{Mobility Plus} grant 911/MOB/2012/0.}

\begin{document}

\begin{abstract}
Let $S$ be a subset of $\mathbb{R}^d$ with finite positive Lebesgue measure.
The \emph{Beer index of convexity} $\operatorname{b}(S)$ of $S$ is the probability that two points of $S$ chosen uniformly independently at random see each other in~$S$.
The \emph{convexity ratio} $\operatorname{c}(S)$ of $S$ is the Lebesgue measure of the largest convex subset of $S$ divided by the Lebesgue measure of~$S$.
We investigate the relationship between these two natural measures of convexity.

We show that every set $S\subseteq\mathbb{R}^2$ with simply connected components satisfies $\operatorname{b}(S)\leq\alpha\operatorname{c}(S)$ for an absolute constant $\alpha$, provided $\operatorname{b}(S)$ is defined.
This implies an affirmative answer to the conjecture of Cabello et~al.\ that this estimate holds for simple polygons.

We also consider higher-order generalizations of $\operatorname{b}(S)$.
For $1\leq k\leq d$, the \emph{$k$-index of convexity} $\operatorname{b}_k(S)$ of a set $S\subseteq\mathbb{R}^d$ is the probability that the convex hull of a $(k+1)$-tuple of points chosen uniformly independently at random from $S$ is contained in~$S$.
We show that for every $d\geq 2$ there is a constant $\beta(d)>0$ such that every set $S\subseteq\mathbb{R}^d$ satisfies $\operatorname{b}_d(S)\leq\beta\operatorname{c}(S)$, provided $\operatorname{b}_d(S)$ exists.
We provide an almost matching lower bound by showing that there is a constant $\gamma(d)>0$ such that for every $\varepsilon\in(0,1)$ there is a set $S\subseteq\mathbb{R}^d$ of Lebesgue measure $1$ satisfying $\operatorname{c}(S)\leq\varepsilon$ and $\operatorname{b}_d(S)\geq\gamma\frac{\varepsilon}{\log_2{1/\varepsilon}}\geq\gamma\frac{\operatorname{c}(S)}{\log_2{1/\operatorname{c}(S)}}$.
\end{abstract}

\maketitle

\section{Introduction}

For positive integers $k$ and $d$ and a Lebesgue measurable set $S\subseteq\bbR^d$, we use $\lambda_k(S)$ to denote the $k$-dimensional Lebesgue measure of~$S$.
We omit the subscript $k$ when it is clear from the context.
We also write ``measure'' instead of ``Lebesgue measure'', as we do not use any other measure in the paper.

For a set $S\subseteq\bbR^d$, let $\smc(S)$ denote the supremum of the measures of convex subsets of~$S$.
Since all convex subsets of $\bbR^d$ are measurable \cite{Lan86}, the value of $\smc(S)$ is well defined.
Moreover, Goodman's result \cite{Goo81} implies that the supremum is achieved on compact sets $S$, hence it can be replaced by maximum in this case.
When $S$ has finite positive measure, let $\conv(S)$ be defined as $\smc(S)/\lambda_d(S)$.
We call the parameter $\conv(S)$ the \emph{convexity ratio} of~$S$.

For two points $A,B\in\bbR^d$, let $\overline{AB}$ denote the closed line segment with endpoints $A$ and~$B$. Let $S$ be a subset of~$\bbR^d$.
We say that points $A,B\in S$ are \emph{visible} one from the other or \emph{see} each other in $S$ if the line segment $\overline{AB}$ is contained in~$S$.
For a point $A\in S$, we use $\Vis(A,S)$ to denote the set of points that are visible from $A$ in~$S$.
More generally, for a subset $T$ of $S$, we use $\Vis(T,S)$ to denote the set of points that are visible in $S$ from~$T$.
That is, $\Vis(T,S)$ is the set of points $A\in S$ for which there is a point $B\in T$ such that $\overline{AB}\subseteq S$.

Let $\Seg(S)$ denote the set $\{(A,B)\in S\times S\colon\overline{AB}\subseteq S\}\subseteq(\bbR^d)^2$, which we call the \emph{segment set} of~$S$.
For a set $S\subseteq\bbR^d$ with finite positive measure and with measurable $\Seg(S)$, we define the parameter $\beer(S)\in[0,1]$ by
\begin{equation*}
\beer(S)\colonequals\frac{\lambda_{2d}(\Seg(S))}{\lambda_d(S)^2}.
\end{equation*}
If $S$ is not measurable, or if its measure is not positive and finite, or if $\Seg(S)$ is not measurable, we leave $\beer(S)$ undefined.
Note that if $\beer(S)$ is defined for a set $S$, then $\conv(S)$ is defined as well.

We call $\beer(S)$ the \emph{Beer index of convexity} (or just \emph{Beer index}) of~$S$.
It can be interpreted as the probability that two points $A$ and $B$ of $S$ chosen uniformly independently at random see each other in~$S$.

\subsection{Previous results}

The Beer index was introduced in the 1970s by Beer \cite{Bee73a,Bee73b,Bee74}, who called it ``the index of convexity''.
Beer was motivated by studying the continuity properties of $\lambda(\Vis(A,S))$ as a function of~$A$.
For polygonal regions, an equivalent parameter was later independently defined by Stern \cite{Ste89}, who called it ``the degree of convexity''.
Stern was motivated by the problem of finding a computationally tractable way to quantify how close a given set is to being convex.
He showed that the Beer index of a polygon $P$ can be approximated by a Monte Carlo estimation.
Later, Rote \cite{Rot13} showed that for a polygonal region $P$ with $n$ edges the Beer index can be evaluated in polynomial time as a sum of $O(n^9)$ closed-form expressions.

Cabello et~al.\ \cite{CCK+14} have studied the relationship between the Beer index and the convexity ratio, and applied their results in the analysis of their near-linear-time approximation algorithm for finding the largest convex subset of a polygon.
We describe some of their results in more detail in Section \ref{ssec:results}.

\subsection{Terminology and notation}

We assume familiarity with basic topological notions such as path-connectedness, simple connectedness, Jordan curve, etc.
The reader can find these definitions, for example, in Prasolov's book \cite{Pra06}.

Let $\bd S$, $\inte{S}$, and $\overline{S}$ denote the boundary, the interior, and the closure of a set $S$, respectively.
For a point $A\in\bbR^2$ and $\varepsilon>0$, let $\nei{\varepsilon}{A}$ denote the open disc centered at $A$ with radius~$\varepsilon$.
For a set $X\subseteq\bbR^2$ and $\varepsilon>0$, let $\nei{\varepsilon}{X}\colonequals\bigcup_{A\in X}\nei{\varepsilon}{A}$.
A \emph{neighborhood} of a point $A\in\bbR^2$ or a set $X\subseteq\bbR^2$ is a set of the form $\nei{\varepsilon}{A}$ or $\nei{\varepsilon}{X}$, respectively, for some $\varepsilon>0$.

A closed interval with endpoints $a$ and $b$ is denoted by $[a,b]$.
Intervals $[a,b]$ with $a>b$ are considered empty.
For a point $A\in\bbR^2$, we use $x(A)$ and $y(A)$ to denote the $x$-coordinate and the $y$-coordinate of $A$, respectively.

A \emph{polygonal curve} $\Gamma$ in $\bbR^d$ is a curve specified by a sequence $(A_1,\dotsc,A_n)$ of points of $\bbR^d$ such that $\Gamma$ consists of the line segments connecting the points $A_i$ and $A_{i+1}$ for $i=1,\dotsc,n-1$.
If $A_1=A_n$, then the polygonal curve $\Gamma$ is \emph{closed}.
A polygonal curve that is not closed is called a \emph{polygonal line}.

A set $X\subseteq\bbR^2$ is \emph{polygonally connected}, or \emph{p-connected} for short, if any two points of $X$ can be connected by a polygonal line in $X$, or equivalently, by a self-avoiding polygonal line in~$X$.
For a set $X$, the relation ``$A$ and $B$ can be connected by a polygonal line in $X$'' is an equivalence relation on $X$, and its equivalence classes are the \emph{p-components} of~$X$.
A set $S$ is \emph{p-componentwise simply connected} if every p-component of $S$ is simply connected.

A \emph{line segment} in $\bbR^d$ is a bounded convex subset of a line.
A \emph{closed line segment} includes both endpoints, while an \emph{open line segment} excludes both endpoints.
For two points $A$ and $B$ in $\bbR^d$, we use $AB$ to denote the open line segment with endpoints $A$ and~$B$.
A closed line segment with endpoints $A$ and $B$ is denoted by $\overline{AB}$.

We say that a set $S\subseteq\bbR^d$ is \emph{star-shaped} if there is a point $C\in S$ such that $\Vis(C,S)=S$.
That is, a star-shaped set $S$ contains a point which sees the entire~$S$.
Similarly, we say that a set $S$ is \emph{weakly star-shaped} if $S$ contains a line segment $\ell$ such that $\Vis(\ell,S)=S$.

\subsection{Results}
\label{ssec:results}

We start with a few simple observations.
Let $S$ be a subset of $\bbR^2$ such that $\Seg(S)$ is measurable.
For every $\varepsilon>0$, $S$ contains a convex subset $K$ of measure at least $(\conv(S)-\varepsilon)\lambda_2(S)$.
Two points of $S$ chosen uniformly independently at random both belong to $K$ with probability at least $(\conv(S)-\varepsilon)^2$, hence $\beer(S)\geq(\conv(S)-\varepsilon)^2$.
This yields $\beer(S)\geq\conv(S)^2$.
This simple lower bound on $\beer(S)$ is tight, as shown by a set $S$ which is a disjoint union of a single large convex component and a large number of small components of negligible size.

It is more challenging to find an upper bound on $\beer(S)$ in terms of $\conv(S)$, possibly under additional assumptions on the set~$S$.
This is the general problem addressed in this paper.

As a motivating example, observe that a set $S$ consisting of $n$ disjoint convex components of the same size satisfies $\beer(S)=\conv(S)=\frac{1}{n}$.
It is easy to modify this example to obtain, for any $\varepsilon>0$, a simple star-shaped polygon $P$ with $\beer(P)\geq\frac{1}{n}-\varepsilon$ and $\conv(P)\leq\frac{1}{n}$, see Figure~\ref{fig:lower-bound-2d}.
This shows that $\beer(S)$ cannot be bounded from above by a sublinear function of $\conv(S)$, even for simple polygons~$S$.

\begin{figure}[t]
\centering
\includegraphics{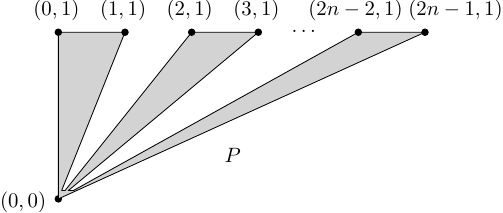}
\caption{A star-shaped polygon $P$ with $\beer(P)\geq\frac{1}{n}-\varepsilon$ and $\conv(P)\leq\frac{1}{n}$.
The polygon $P$ is a union of $n$ triangles $(0,0)(2i,1)(2i+1,1)$, $i=0,\dotsc,n-1$, and of a triangle $(0,0)(0,\delta)((2n-1)\delta,\delta)$, where $\delta$ is very small.}
\label{fig:lower-bound-2d}
\end{figure}

For weakly star-shaped polygons, Cabello et~al.\ \cite{CCK+14} showed that the above example is essentially optimal, providing the following linear upper bound on $\beer(S)$.

\begin{theorem}[{\cite[Theorem 5]{CCK+14}}]
\label{thm:previous-star-shaped}
For every weakly star-shaped simple polygon\/ $P$, we have\/ $\beer(P)\leq 18\conv(P)$.
\end{theorem}

For polygons that are not weakly star-shaped, Cabello et~al.\ \cite{CCK+14} gave a superlinear bound.

\begin{theorem}[{\cite[Theorem 6]{CCK+14}}]
Every simple polygon\/ $P$ satisfies
\begin{equation*}
\beer(P)\leq 12\conv(P)\left(1+\log_2{\frac{1}{\conv(P)}}\right).
\end{equation*}
\end{theorem}

Moreover, Cabello et~al.\ \cite{CCK+14} conjectured that even for a general simple polygon $P$, $\beer(P)$ can be bounded from above by a linear function of $\conv(P)$.
(The question whether $\beer(P)=O(\conv(P))$ for simple polygons $P$ was originally asked by Cabello and Saumell \cite{CaS-personal}.)
The next theorem, which is the first main result of this paper, verifies this conjecture.
Recall that $\beer(S)$ is defined for a set $S$ if and only if $S$ has finite positive measure and $\Seg(S)$ is measurable.
Recall also that a set is p-componentwise simply connected if its p-components are simply connected.
In particular, every simply connected set is p-componentwise simply connected.

\begin{theorem}
\label{thm:upper-bound-2d}
Every p-componentwise simply connected set\/ $S\subseteq\bbR^2$ whose\/ $\beer(S)$ is defined satisfies\/ $\beer(S)\leq 180\conv(S)$.
\end{theorem}

Clearly, every simple polygon satisfies the assumptions of Theorem \ref{thm:upper-bound-2d}.
Hence we directly obtain the following, which verifies the conjecture of Cabello et~al.\ \cite{CCK+14}.

\begin{corollary}
\label{cor:upper-bound-2d}
Every simple polygon\/ $P\subseteq\bbR^2$ satisfies\/ $\beer(P)\leq 180\conv(P)$.
\end{corollary}

The main restriction in Theorem \ref{thm:upper-bound-2d} is the assumption that $S$ is p-componentwise simply connected.
This assumption cannot be omitted, as shown by the set $S\colonequals[0,1]^2\setminus\bbQ^2$, where it is easy to verify that $\conv(S)=0$ and $\beer(S)=1$, see Proposition \ref{prop:example}.

A related construction shows that Theorem \ref{thm:upper-bound-2d} fails in higher dimensions.
To see this, consider again the set $S\colonequals[0,1]^2\setminus\bbQ^2$, and define a set $S'\subseteq\bbR^3$ by
\begin{equation*}
S'\colonequals\{(tx,ty,t)\colon t\in[0,1]\text{ and }(x,y)\in S\}.
\end{equation*}
Again, it is easy to verify that $\conv(S')=0$ and $\beer(S')=1$, although $S'$ is simply connected, even star-shaped.

Despite these examples, we will show that meaningful analogues of Theorem \ref{thm:upper-bound-2d} for higher dimensions and for sets that are not p-componentwise simply connected are possible.
The key is to use higher-order generalizations of the Beer index, which we introduce now.

For $k\in\{1,\dotsc,d\}$ and a set $S\subseteq\bbR^d$, we define the set $\Simp_k(S)\subseteq(\bbR^d)^{k+1}$ by
\begin{equation*}
\Simp_k(S)\colonequals\{(A_0,\dotsc,A_k)\in S^{k+1}\colon\Conv(\{A_0,\dotsc,A_k\})\subseteq S\},
\end{equation*}
where the operator $\Conv$ denotes the convex hull of a set of points.
We call $\Simp_k(S)$ the \emph{$k$-simplex set} of~$S$.
Note that $\Simp_1(S)=\Seg(S)$.

For $k\in\{1,\dotsc,d\}$ and a set $S\subseteq\bbR^d$ with finite positive measure and with measurable $\Simp_k(S)$, we define $\beer_k(S)$ by
\begin{equation*}
\beer_k(S)\colonequals\frac{\lambda_{(k+1)d}(\Simp_k(S))}{\lambda_d(S)^{k+1}}.
\end{equation*}
Note that $\beer_1(S)=\beer(S)$.
We call $\beer_k(S)$ the \emph{$k$-index of convexity} of~$S$.
We again leave $\beer_k(S)$ undefined if $S$ or $\Simp_k(S)$ is non-measurable, or if the measure of $S$ is not finite and positive.

We can view $\beer_k(S)$ as the probability that the convex hull of $k+1$ points chosen from $S$ uniformly independently at random is contained in~$S$.
For any $S\subseteq\bbR^d$, we have $\beer_1(S)\geq\beer_2(S)\geq\dotsb\geq\beer_d(S)$, provided all the $\beer_k(S)$ are defined.

We remark that the set $S\colonequals[0,1]^d\setminus\bbQ^d$ satisfies $\conv(S)=0$ and $\beer_1(S)=\beer_2(S)=\dotsb=\beer_{d-1}(S)=1$, see Proposition \ref{prop:example}.
Thus, for a general set $S\subseteq\bbR^d$, only the $d$-index of convexity can conceivably admit a nontrivial upper bound in terms of $\conv(S)$.
Our next result shows that such an upper bound on $\beer_d(S)$ exists and is linear in $\conv(S)$.

\begin{theorem}
\label{thm:upper-bound-full-vis}
For every\/ $d\geq 2$, there is a constant\/ $\beta=\beta(d)>0$ such that every set\/ $S\subseteq\bbR^d$ with\/ $\beer_d(S)$ defined satisfies\/ $\beer_d(S)\leq\beta\conv(S)$.
\end{theorem}

We do not know if the linear upper bound in Theorem \ref{thm:upper-bound-full-vis} is best possible.
We can, however, construct examples showing that the bound is optimal up to a logarithmic factor.
This is our last main result.

\begin{theorem}
\label{thm:lower-bound}
For every\/ $d\geq 2$, there is a constant\/ $\gamma=\gamma(d)>0$ such that for every\/ $\varepsilon\in(0,1)$, there is a set\/ $S\subseteq\bbR^d$ satisfying\/ $\conv(S)\leq\varepsilon$ and\/ $\smash[b]{\beer_d(S)\geq\gamma\frac{\varepsilon}{\log_2 1/\varepsilon}}$, and in particular, we have\/ $\beer_d(S)\geq\gamma\frac{\conv(S)}{\log_2 1/\conv(S)}$.
\end{theorem}

The proof of Theorem \ref{thm:upper-bound-2d} is given in Section~\ref{sec:visibility-2d}.
In Section~\ref{sec:visibility-higher-dim}, we will prove Theorems \ref{thm:upper-bound-full-vis} and \ref{thm:lower-bound}.
We conclude, in Section~\ref{sec:generalizations}, with some further remarks and a collection of open problems.

\section{Bounding the mutual visibility in the plane}
\label{sec:visibility-2d}

The goal of this section is to prove Theorem \ref{thm:upper-bound-2d}.
Since the proof is rather long and complicated, we first present a high-level overview of its main ideas.

We first show that it is sufficient to prove the estimate from Theorem \ref{thm:upper-bound-2d} for bounded open simply connected sets.
This is formalized by the next lemma, whose proof can be found in Section \ref{ssec:reduction}.

\begin{lemma}
\label{lem:reduction}
Let\/ $\alpha>0$ be a constant such that every bounded open simply connected set\/ $T\subseteq\bbR^2$ satisfies\/ $\beer(T)\leq\alpha\conv(T)$.
It follows that every p-componentwise simply connected set\/ $S\subseteq\bbR^2$ with\/ $\beer(S)$ defined satisfies\/ $\beer(S)\leq\alpha\conv(S)$.
\end{lemma}

In the proof of Lemma \ref{lem:reduction}, we first show that the set $S$ can be reduced to a bounded open set $S''$ whose Beer index $\beer(S'')$ can be arbitrarily close to $\beer(S)$ from below.
This is done by considering a part $S'$ of $S$ that is contained in a sufficiently large disc and by showing that all segments in $S'$ are in fact contained in the interior of $S'$, except for a set of measure zero.
The proof is then finished by choosing $S''$ as the interior of $S'$ and by applying the assumption of the lemma to every p-component of~$S''$.

Suppose now that $S$ is a bounded open simply connected set.
We seek a bound of the form $\beer(S)=O(\conv(S))$.
This is equivalent to a bound of the form $\lambda_4(\Seg(S))=O(\smc(S)\lambda_2(S))$.
We therefore need a suitable upper bound on $\lambda_4(\Seg(S))$.

We first choose in $S$ a \emph{diagonal} $\ell$ (i.e., an inclusion-maximal line segment in $S$), and show that the set $S\setminus\ell$ is a union of two open simply connected sets $S_1$ and $S_2$ (Lemma \ref{lem:cutting}).
It is not hard to show that the segments in $S$ that cross the diagonal $\ell$ contribute to $\lambda_4(\Seg(S))$ by at most $O(\smc(S)\lambda_2(S))$ (Lemma \ref{lem:crossing-visibility}).
Our main task is to bound the measure of $\Seg(S_i\cup\ell)$ for $i=1,2$.
The two sets $S_i\cup\ell$ are what we call \emph{rooted sets}.
Informally, a rooted set is a union of a simply connected open set $S'$ and an open segment $r\subseteq\bd S'$, called the root.

To bound $\lambda_4(\Seg(R))$ for a rooted set $R$ with root $r$, we partition $R$ into \emph{levels} $L_1,L_2,\dotsc$, where $L_k$ contains the
points of $R$ that can be connected to $r$ by a polygonal line with $k$ segments, but not by a polygonal line with $k-1$ segments.
Each segment in $R$ is contained in a union $L_i\cup L_{i+1}$ for some $i\geq 1$.
Thus, a bound of the form $\lambda_4(\Seg(L_i\cup L_{i+1}))=O(\smc(R)\lambda_2(L_i\cup L_{i+1}))$ implies the required bound for $\lambda_4(\Seg(R))$.

We will show that each p-component of $L_i\cup L_{i+1}$ is a rooted set, with the extra property that all its points are reachable from its root by a polygonal line with at most two segments (Lemma \ref{lem:branch-structure}).
To handle such sets, we will generalize the techniques that Cabello et~al.\ \cite{CCK+14} have used to handle weakly star-shaped sets in their proof of Theorem~\ref{thm:previous-star-shaped}.
We will assign to every point $A\in R$ a set $\fT(A)$ of measure $O(\smc(R))$, such that for every $(A,B)\in\Seg(R)$, we have either $B\in\fT(A)$ or $A\in\fT(B)$ (Lemma~\ref{lem:rooted-cover}).
From this, Theorem \ref{thm:upper-bound-2d} will follow easily.

\subsection{Proof of Theorem \ref{thm:upper-bound-2d} for bounded open simply connected sets}

First, we need a few auxiliary lemmas.

\begin{lemma}
\label{lem:open}
For every positive integer\/ $d$, if\/ $S$ is an open subset of\/ $\bbR^d$, then the set\/ $\Seg(S)$ is open and the set\/ $\Vis(A,S)$ is open for every point\/ $A\in S$.
\end{lemma}

\begin{proof}
Choose a pair of points $(A,B)\in\Seg(S)$.
Since $S$ is open and $\overline{AB}$ is compact, there is $\varepsilon>0$ such that $\nei{\varepsilon}{\overline{AB}}\subseteq S$.
Consequently, for any $A'\in\nei{\varepsilon}{A}$ and $B'\in\nei{\varepsilon}{B}$, we have $\overline{A'B'}\subseteq S$, that is, $(A',B')\in\Seg(S)$.
This shows that the set $\Seg(S)$ is open.
If we fix $A'=A$, then it follows that the set $\Vis(A,S)$ is open.
\end{proof}

\begin{lemma}
\label{lem:visibility}
Let\/ $S$ be a simply connected subset of\/ $\bbR^2$ and let\/ $\ell$ and\/ $\ell'$ be line segments in\/~$S$.
It follows that the set\/ $\Vis(\ell',S)\cap\ell$ is a (possibly empty) subsegment of\/~$\ell$.
\end{lemma}

\begin{proof}
The statement is trivially true if $\ell$ and $\ell'$ intersect or have the same supporting line, or if $\Vis(\ell',S)\cap\ell$ is empty.
Suppose that these situations do not occur.
Let $A,B\in\ell$ and $A',B'\in\ell'$ be such that $\overline{AA'},\overline{BB'}\subseteq S$.
The points $A,A',B',B$ form a (possibly self-intersecting) tetragon $Q$ whose boundary is contained in~$S$.
Since $S$ is simply connected, the interior of $Q$ is contained in~$S$.
If $Q$ is not self-intersecting, then clearly $\overline{AB}\subseteq\Vis(\ell',S)$.
Otherwise, $\overline{AA'}$ and $\overline{BB'}$ have a point $D$ in common, and every point $C\in AB$ is visible in $R$ from the point $C'\in A'B'$ such that $D\in\overline{CC'}$.
This shows that $\Vis(\ell',S)\cap\ell$ is a convex subset and hence a subsegment of~$\ell$.
\end{proof}

Now, we define rooted sets and their tree-structured decomposition, and we explain how they arise in the proof of Theorem \ref{thm:upper-bound-2d}.

A set $S\subseteq\bbR^2$ is \emph{half-open} if every point $A\in S$ has a neighborhood $\nei{\varepsilon}{A}$ that satisfies one of the following two conditions:
\begin{enumerate}
\item\label{item:half-open-1} $\nei{\varepsilon}{A}\subseteq S$,
\item\label{item:half-open-2} $\nei{\varepsilon}{A}\cap\bd S$ is a diameter of $\nei{\varepsilon}{A}$ splitting it into two subsets, one of which (including the diameter) is $\nei{\varepsilon}{A}\cap S$ and the other (excluding the diameter) is $\nei{\varepsilon}{A}\setminus S$.
\end{enumerate}
The condition~\ref{item:half-open-1} holds for points $A\in\inte{S}$, while the condition~\ref{item:half-open-2} holds for points $A\in\bd S$.

A set $R\subseteq\bbR^2$ is a \emph{rooted set} if the following conditions are satisfied:
\begin{enumerate}
\item $R$ is bounded,
\item $R$ is p-connected and simply connected,
\item $R$ is half-open,
\item $R\cap\bd R$ is an open line segment.
\end{enumerate}
The open line segment $R\cap\bd R$ is called the \emph{root} of~$R$.
Every rooted set, as the union of a non-empty open set and an open line segment, is measurable and has positive measure.

A \emph{diagonal} of a set $S\subseteq\bbR^2$ is a line segment contained in $S$ that is not a proper subset of any other line segment contained in~$S$.
Clearly, if $S$ is open, then every diagonal of $S$ is an open line segment.
It is easy to see that the root of a rooted set $R$ is a diagonal of~$R$.

The following lemma allows us to use a diagonal to split a bounded open simply connected subset of $\bbR^2$ into two rooted sets.
It is intuitively clear, and its formal proof is postponed to Section \ref{ssec:cutting}.

\begin{lemma}
\label{lem:cutting}
Let\/ $S$ be a bounded open simply connected subset of\/ $\bbR^2$, and let\/ $\ell$ be a diagonal of\/~$S$.
It follows that the set\/ $S\setminus\ell$ has two p-components\/ $S_1$ and\/~$S_2$.
Moreover, $S_1\cup\ell$ and $S_2\cup\ell$ are rooted sets, and\/ $\ell$ is their common root.
\end{lemma}

Let $R$ be a rooted set.
For a positive integer $k$, the \emph{$k$th level} $L_k$ of $R$ is the set of points of $R$ that can be connected to the root of $R$ by a polygonal line in $R$ consisting of $k$ segments but cannot be connected to the root of $R$ by a polygonal line in $R$ consisting of fewer than $k$ segments.
We consider a degenerate one-vertex polygonal line as consisting of one degenerate segment, so the root of $R$ is part of~$L_1$.
Thus $L_1=\Vis(r,R)$, where $r$ denotes the root of~$R$.
A \emph{$k$-body} of $R$ is a p-component of~$L_k$.
A \emph{body} of $R$ is a $k$-body of $R$ for some~$k$.
See Figure~\ref{fig:segment} for an example of a rooted set and its partitioning into levels and bodies.

We say that a rooted set $P$ is \emph{attached} to a set $Q\subseteq\bbR^2\setminus P$ if the root of $P$ is subset of the interior of $P\cup Q$.
The following lemma explains the structure of levels and bodies.
Although it is intuitively clear, its formal proof requires quite a lot of work and can be found in Section \ref{ssec:branch-structure}.

\begin{lemma}
\label{lem:branch-structure}
Let\/ $R$ be a rooted set and\/ $(L_k)_{k\geq 1}$ be its partition into levels.
It follows that
\begin{enumerate}
\item\label{item:body-1} $R=\bigcup_{k\geq 1}L_k$; consequently, $R$ is the union of all its bodies;
\item\label{item:body-2} every body\/ $P$ of\/ $R$ is a rooted set such that\/ $P=\Vis(r,P)$, where\/ $r$ denotes the root of\/~$P$;
\item\label{item:body-3} $L_1$ is the unique\/ $1$-body of\/ $R$, and the root of\/ $L_1$ is the root of\/~$R$;
\item\label{item:body-4} every\/ $j$-body\/ $P$ of\/ $R$ with\/ $j\geq 2$ is attached to a unique\/ $(j-1)$-body of\/~$R$.
\end{enumerate}
\end{lemma}

Lemma \ref{lem:branch-structure} yields a tree structure on the bodies of~$R$.
The root of this tree is the unique $1$-body $L_1$ of $R$, called the \emph{root body} of~$R$.
For a $k$-body $P$ of $R$ with $k\geq 2$, the parent of $P$ in the tree is the unique $(k-1)$-body of $R$ that $P$ is attached to, called the \emph{parent body} of~$P$.

\begin{figure}[t]
\centering
\includegraphics{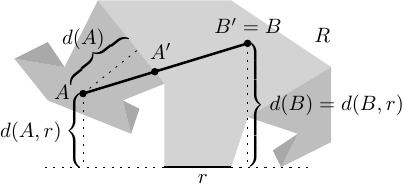}
\caption{Example of a rooted set $R$ partitioned into six bodies. The three levels of $R$ are distinguished with three shades of gray.
The segment $A'B'\cup\{B'\}$ is the base segment of $\overline{AB}$.}
\label{fig:segment}
\end{figure}

\begin{lemma}
\label{lem:segment}
Let\/ $R$ be a rooted set, $(L_k)_{k\geq 1}$ be the partition of\/ $R$ into levels, $\ell$ be a closed line segment in\/ $R$, and\/ $k\geq 1$ be minimum such that\/ $\ell\cap L_k\neq\emptyset$.
It follows that\/ $\ell\subseteq L_k\cup L_{k+1}$, $\ell\cap L_k$ is a subsegment of\/ $\ell$ contained in a single\/ $k$-body\/ $P$ of\/ $R$, and\/ $\ell\cap L_{k+1}$ consists of at most two subsegments of\/ $\ell$ each contained in a single\/ $(k+1)$-body whose parent body is\/~$P$.
\end{lemma}

\begin{proof}
The definition of the levels directly yields $\ell\subseteq L_k\cup L_{k+1}$.
The segment $\ell$ splits into subsegments each contained in a single $k$-body or $(k+1)$-body of~$R$.
By Lemma~\ref{lem:branch-structure}, the bodies of any two consecutive of these subsegments are in the parent-child relation of the body tree.
This implies that $\ell\cap L_k$ lies within a single $k$-body~$P$.
By Lemma \ref{lem:visibility}, $\ell\cap L_k$ is a subsegment of~$\ell$.
Consequently, $\ell\cap L_{k+1}$ consists of at most two subsegments.
\end{proof}

In the setting of Lemma \ref{lem:segment}, we call the subsegment $\ell\cap L_k$ of $\ell$ the \emph{base segment} of $\ell$, and we call the body $P$ that contains $\ell\cap L_k$ the \emph{base body} of~$\ell$.
See Figure~\ref{fig:segment} for an example.

The following lemma is the crucial part of the proof of Theorem \ref{thm:upper-bound-2d}.

\begin{lemma}
\label{lem:rooted-cover}
If\/ $R$ is a rooted set, then every point\/ $A\in R$ can be assigned a measurable set\/ $\fT(A)\subseteq\bbR^2$ so that the following is satisfied:
\begin{enumerate}
\item $\lambda_2(\fT(A))<87\smc(R)$;
\item for every line segment\/ $\overline{BC}$ in\/ $R$, we have either\/ $B\in\fT(C)$ or\/ $C\in\fT(B)$;
\item the set\/ $\{(A,B)\colon A\in R$ and\/ $B\in\fT(A)\}$ is measurable.
\end{enumerate}
\end{lemma}

\begin{proof}
Let $P$ be a body of $R$ with the root~$r$.
First, we show that $P$ is entirely contained in one closed half-plane defined by the supporting line of~$r$.
Let $h^-$ and $h^+$ be the two open half-planes defined by the supporting line of~$r$.
According to the definition of a rooted set, the sets $\{D\in r\colon\exists\varepsilon>0\colon\nei{\varepsilon}{D}\cap h^-=\nei{\varepsilon}{D}\cap(P\setminus r)\}$ and $\{D\in r\colon\exists\varepsilon>0\colon\nei{\varepsilon}{D}\cap h^+=\nei{\varepsilon}{D}\cap(P\setminus r)\}$ are open and partition the entire $r$, hence one of them must be empty.
This implies that the segments connecting $r$ to $P\setminus r$ lie all in $h^-$ or all in~$h^+$.
Since $P=\Vis(r,P)$, we conclude that $P\subseteq h^-$ or $P\subseteq h^+$.

According to the above, we can rotate and translate the set $R$ so that $r$ lies on the $x$-axis and $P$ lies in the half-plane $\{B\in\bbR^2\colon y(B)\geq 0\}$.
For a point $A\in R$, we use $d(A,r)$ to denote the $y$-coordinate of $A$ after such a rotation and translation of~$R$.
We use $d(A)$ to denote $d(A,r)$ where $r$ is the root of the body of~$A$.
It follows that $d(A)\geq 0$ for every $A\in R$.

Let $\gamma\in(0,1)$ be a fixed constant whose value will be specified at the end of the proof.
For a point $A\in R$, we define sets
\begin{align*}
\fV_1(A)&\colonequals\{B\in\Vis(A,R)\colon|A'B'|\geq\gamma|AB|,\:A\in\Vis(r'',R),\:d(A',r'')\geq d(B',r'')\},\\
\fV_2(A)&\colonequals\{B\in\Vis(A,R)\colon|A'B'|\geq\gamma|AB|,\:A\notin\Vis(r'',R),\:d(A',r'')\geq d(B',r'')\},\\
\fV_3(A)&\colonequals\{B\in\Vis(A,R)\colon|A'B'|<\gamma|AB|,\:|AA'|\geq|BB'|\},
\end{align*}
where $r''$ denotes the root of the base body of $\overline{AB}$ and $A'$ and $B'$ denote the endpoints of the base segment of $\overline{AB}$ such that $|AA'|<|AB'|$.
For every $A\in R$, the sets $\fV_1(A)$, $\fV_2(A)$, and $\fV_3(A)$ are pairwise disjoint.
Moreover, we have $A\in\bigcup_{i=1}^3\fV_i(B)$ or $B\in\bigcup_{i=1}^3\fV_i(A)$ for every line segment $\overline{AB}$ in~$R$.
If for some $B\in\bigcup_{i=1}^3\fV_i(A)$ the point $A$ lies on $r''$, then we have $B\in\fV_1(A)$ and $\fV_1(A)\subseteq r''$.

For the rest of the proof, we fix a point $A\in R$.
We show that the union $\bigcup_{i=1}^3\fV_i(A)$ is contained in a measurable set $\fT(A)\subseteq\bbR^2$ with $\lambda_2(\fT(A))<87\smc(R)$ that is a union of three trapezoids.
We let $P$ be the body of $A$ and $r$ be the root of~$P$.
If $P$ is a $k$-body with $k\geq 2$, then we use $r'$ to denote the root of the parent body of~$P$.

\begin{claim}
\label{claim1}
\itshape
$\fV_1(A)$ is contained in a trapezoid\/ $\fT_1(A)$ with area\/ $6\gamma^{-2}\smc(R)$.
\end{claim}

Let $H$ be a point of $r$ such that $\overline{AH}\subseteq R$.
Let $T'$ be the $r$-parallel trapezoid of height $d(A)$ with bases of length $\smash[b]{\frac{8\smc(R)}{d(A)}}$ and $\smash[b]{\frac{4\smc(R)}{d(A)}}$ such that $A$ is the center of the larger base and $H$ is the center of the smaller base.
The homothety with center $A$ and ratio $\gamma^{-1}$ transforms $T'$ into the trapezoid $T\colonequals A+\gamma^{-1}(T'-A)$.
Since the area of $T'$ is $6\smc(R)$, the area of $T$ is $6\gamma^{-2}\smc(R)$.
We show that $\fV_1(A)\subseteq T$.
See Figure~\ref{fig:claim1} for an illustration.

\begin{figure}[t]
\centering
\includegraphics{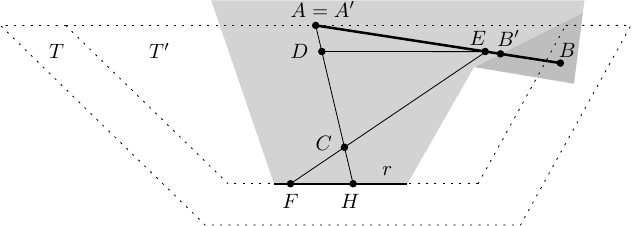}
\caption{Illustration for the proof of Claim~\ref{claim1} in the proof of Lemma \ref{lem:rooted-cover}.}
\label{fig:claim1}
\end{figure}

Let $B$ be a point in $\fV_1(A)$.
Using a similar approach to the one used by Cabello et~al.\ \cite{CCK+14} in the proof of Theorem \ref{thm:previous-star-shaped}, we show that $B\in T$.
Let $A'B'$ be the base segment of $\overline{AB}$ such that $|AA'|<|AB'|$.
Since $B\in\fV_1(A)$, we have $|A'B'|\geq\gamma|AB|$, $A\in\Vis(r'',R)$, and $d(B,r'')\leq d(A,r'')$, where $r''$ denotes the root of the base level of $\overline{AB}$.
Since $A$ is visible from $r''$ in $R$, the base body of $\overline{AB}$ is the body of $A$ and thus $A=A'$ and $r=r''$.
As we have observed, every point $C\in\{A\}\cup AB'$ satisfies $d(C,r)=d(C)\geq 0$.

Let $\varepsilon>0$.
There is a point $E\in AB'$ such that $|B'E|<\varepsilon$.
Since $E$ lies on the base segment of $\overline{AB}$, there is $F\in r$ such that $\overline{EF}\subseteq R$.
It is possible to choose $F$ so that $\overline{AH}$ and $\overline{EF}$ have a point $C$ in common where $C\neq F,H$.
Let $D$ be a point of $\overline{AH}$ with $d(D)=d(E)$.
The point $D$ exists, as $d(H)=0\leq d(E)\leq d(A)$.
The points $A,E,F,H$ form a self-intersecting tetragon $Q$ whose boundary is contained in~$R$.
Since $R$ is simply connected, the interior of $Q$ is contained in $R$ and the triangles $ACE$ and $CFH$ have area at most $\smc(R)$.

The triangle $ACE$ is partitioned into triangles $ADE$ and $CDE$ with areas $\frac{1}{2}(d(A)-d(D))|DE|$ and $\frac{1}{2}(d(D)-d(C))|DE|$, respectively.
Therefore, we have $\frac{1}{2}(d(A)-d(C))|DE|=\lambda_2(ACE)\leq\smc(R)$.
This implies
\begin{equation*}
|DE|\leq\frac{2\smc(R)}{d(A)-d(C)}.
\end{equation*}
For the triangle $CFH$, we have $\frac{1}{2}d(C)|FH|=\lambda_2(CFH)\leq\smc(R)$.
By the similarity of the triangles $CFH$ and $CDE$, we have $|FH|=|DE|d(C)/(d(E)-d(C))$ and therefore
\begin{equation*}
|DE|\leq\frac{2\smc(R)}{d(C)^2}(d(E)-d(C)).
\end{equation*}
Since the first upper bound on $|DE|$ is increasing in $d(C)$ and the second is decreasing in $d(C)$, the minimum of the two is maximized when they are equal, that is, when $d(C)=d(A)d(E)/(d(A)+d(E))$.
Then we obtain $|DE|\leq\smash[b]{\frac{2\smc(R)}{d(A)^2}}(d(A)+d(E))$.
This and $0\leq d(E)\leq d(A)$ imply $E\in T'$.
Since $\varepsilon$ can be made arbitrarily small and $T'$ is compact, we have $B'\in T'$.
Since $|AB'|\geq\gamma|AB|$, we conclude that $B\in T$.
This completes the proof of Claim~\ref{claim1}.

\begin{claim}
\label{claim2}
\itshape
$\fV_2(A)$ is contained in a trapezoid\/ $\fT_2(A)$ with area\/ $3(1-\gamma)^{-2}\gamma^{-2}\smc(R)$.
\end{claim}

We assume the point $A$ is not contained in the first level of $R$, as otherwise $\fV_2(A)$ is empty.
Let $p$ be the $r'$-parallel line that contains the point $A$ and let $q$ be the supporting line of~$r$.
Let $p^+$ and $q^+$ denote the closed half-planes defined by $p$ and $q$, respectively, such that $r'\subseteq p^+$ and $A\notin q^+$.
Let $O$ be the intersection point of $p$ and~$q$.

Let $T'\subseteq p^+\cap q^+$ be the trapezoid of height $d(A,r')$ with one base of length $\smash[b]{\frac{4\smc(R)}{(1-\gamma)^2d(A,r')}}$ on $p$, the other base of length $\smash[b]{\frac{2\smc(R)}{(1-\gamma)^2d(A,r')}}$ on the supporting line of $r'$, and one lateral side on~$q$.
The homothety with center $O$ and ratio $\gamma^{-1}$ transforms $T'$ into the trapezoid $T\colonequals O+\gamma^{-1}(T'-O)$.
Since the area of $T'$ is $3(1-\gamma)^{-2}\smc(R)$, the area of $T$ is $3(1-\gamma)^{-2}\gamma^{-2}\smc(R)$.
We show that $\fV_2(A)\subseteq T$.
See Figure~\ref{fig:claim2} for an illustration.

\begin{figure}[t]
\centering
\includegraphics{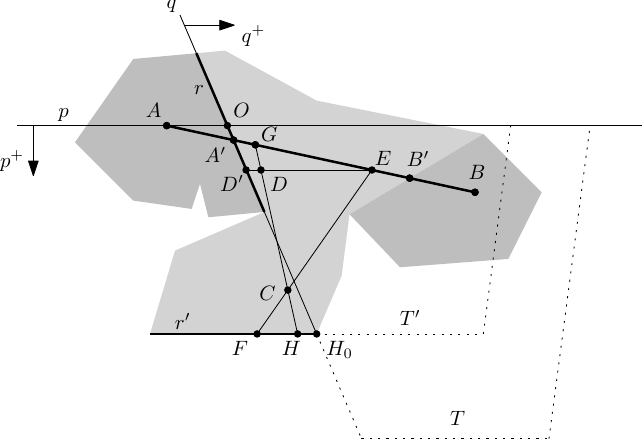}
\caption{Illustration for the proof of Claim~\ref{claim2} in the proof of Lemma \ref{lem:rooted-cover}.}
\label{fig:claim2}
\end{figure}

Let $B$ be a point of $\fV_2(A)$.
We use $A'B'$ to denote the base segment of $\overline{AB}$ such that $|AA'|<|AB'|$.
By the definition of $\fV_2(A)$, we have $|A'B'|\geq\gamma|AB|$, $A\notin\Vis(r'',R)$, and $d(B,r'')\leq d(A,r'')$, where $r''$ denotes the root of the base body of $\overline{AB}$.
By Lemma~\ref{lem:segment} and the fact that $A\notin\Vis(r'',R)$, we have $r'=r''$.
The bound $d(A,r')\geq d(B,r')$ thus implies $A'\in r\cap p^+$ and $B\in q^+$.
We have $d(C,r')=d(C)\geq 0$ for every $C\in A'B'$.

Observe that $(1-\gamma)d(A,r')\leq d(A',r')\leq d(A,r')$.
The upper bound is trivial, as $d(B,r')\leq d(A,r')$ and $A'$ lies on $\overline{AB}$.
For the lower bound, we use the expression $A'=tA+(1-t)B'$ for some $t\in[0,1]$.
This gives us $d(A',r')=td(A,r')+(1-t)d(B',r')$.
By the estimate $|A'B'|\geq\gamma|AB|$, we have
\begin{equation*}
|AA'|+|BB'|\leq(1-\gamma)|AB|=(1-\gamma)(|AB'|+|BB'|).
\end{equation*}
This can be rewritten as $|AA'|\leq(1-\gamma)|AB'|-\gamma|BB'|$.
Consequently, $|BB'|\geq 0$ and $\gamma>0$ imply $|AA'|\leq(1-\gamma)|AB'|$.
This implies $t\geq 1-\gamma$.
Applying the bound $d(B',r')\geq 0$, we conclude that $d(A',r')\geq(1-\gamma)d(A,r')$.

Let $(G_n)_{n\in\bbN}$ be a sequence of points from $A'B'$ that converges to~$A'$.
For every $n\in\bbN$, there is a point $H_n\in r'$ such that $\overline{G_nH_n}\subseteq R$.
Since $\overline{r'}$ is compact, there is a subsequence of $(H_n)_{n\in\bbN}$ that converges to a point $H_0\in\overline{r'}$.
We claim that $H_0\in q$.
Suppose otherwise, and let $q'\neq q$ be the supporting line of $\overline{A'H_0}$.
Let $\varepsilon>0$ be small enough so that $\nei{\varepsilon}{A'}\subseteq R$.
For $n$ large enough, $\overline{G_nH_n}$ is contained in an arbitrarily small neighborhood of~$q'$.
Consequently, for $n$ large enough, the supporting line of $\overline{G_nH_n}$ intersects $q$ at a point $K_n$ such that $\overline{G_nK_n}\subseteq\nei{\varepsilon}{A'}$, which implies $K_n\in r\cap\Vis(r',R)$, a contradiction.

Again, let $\varepsilon>0$.
There is a point $E\in A'B'$ such that $|B'E|<\varepsilon$.
Let $D'$ be a point of $q$ with $d(D',r')=d(E)$, and let $\delta>0$.
There are points $G\in A'B'$ and $H\in r'$ such that $G\in\nei{\delta}{A'}$ and $\overline{GH}\subseteq R\cap\nei{\delta}{q}$.
If $\delta$ is small enough, then $d(E)\leq d(A',r')-\delta\leq d(G)\leq d(A',r')$.
Let $D$ be the point of $\overline{GH}$ with $d(D)=d(E)$.
The point $E$ lies on $A'B'$ and thus it is visible from a point $F\in r'$.
Again, we can choose $F$ so that the line segments $\overline{EF}$ and $\overline{GH}$ have a point $C$ in common where $C\neq F,H$.
The points $E,F,H,G$ form a self-intersecting tetragon $Q$ whose boundary is in~$R$.
The interior of $Q$ is contained in $R$, as $R$ is simply connected.
Therefore, the area of the triangles $CEG$ and $CFH$ is at most $\smc(R)$.

The argument used in the proof of Claim~\ref{claim1} yields
\begin{equation*}
|DE|\leq\frac{2\smc(R)}{d(G)^2}(d(G)+d(E))\leq\frac{2\smc(R)}{(d(A',r')-\delta)^2}(d(A',r')+d(E)).
\end{equation*}
This and the fact that $\delta$ (and consequently $|D'D|$) can be made arbitrarily small yield $|D'E|\leq\frac{2\smc(R)}{d(A',r')^2}(d(A',r')+d(E))$.
This and $d(A',r')\geq(1-\gamma)d(A,r')$ yield $|D'E|\leq\frac{2\smc(R)}{(1-\gamma)^2d(A,r')^2}(d(A,r')+d(E))$.
This and $0\leq d(E)\leq d(A)$ imply $E\in T'$.

Since $\varepsilon$ can be made arbitrarily small and $T'$ is compact, we have $B'\in T'$.
Since $|A'B'|\geq\gamma|AB|\geq\gamma|A'B|$, we conclude that $B\in T$.
This completes the proof of Claim~\ref{claim2}.

\begin{claim}
\label{claim3}
\itshape
$\fV_3(A)$ is contained in a trapezoid\/ $\fT_3(A)$ with area\/ $(4(1-\gamma)^{-2}-1)\smc(R)$.
\end{claim}

By Lemma \ref{lem:visibility}, the points of $r$ that are visible from $A$ in $R$ form a subsegment $CD$ of~$r$.
The homothety with center $A$ and ratio $2(1-\gamma)^{-1}$ transforms the triangle $T'\colonequals ACD$ into the triangle $T''\colonequals A+2(1-\gamma)^{-1}(T'-A)$.
See Figure~\ref{fig:claim3} for an illustration.
We claim that $\fV_3(A)$ is a subset of the trapezoid $T\colonequals T''\setminus T'$.

Let $B$ be an arbitrary point of $\fV_3(A)$.
Consider the segment $\overline{AB}$ with the base segment $A'B'$ such that $|AA'|<|AB'|$.
Since $B\in\fV_3(A)$, we have $|A'B'|<\gamma|AB|$ and $|AA'|\geq|BB'|$.
This implies $|AA'|\geq\frac{1-\gamma}{2}|AB|>0$ and hence $A\neq A'$ and $B\notin P$.
From the definition of $C$ and $D$, we have $A'\in\overline{CD}$.
Since $|AA'|\geq\frac{1-\gamma}{2}|AB|$ and $B\notin P$, we have $B\in T$.

The area of $T$ is $(4(1-\gamma)^{-2}-1)\lambda_2(T')$.
The interior of $T'$ is contained in $R$, as all points of the open segment $CD$ are visible from $A$ in~$R$.
The area of $T'$ is at most $\smc(R)$, as its interior is a convex subset of~$R$.
Consequently, the area of $T$ is at most $(4(1-\gamma)^{-2}-1)\smc(R)$.
This completes the proof of Claim~\ref{claim3}.

\begin{figure}[t]
\centering
\includegraphics{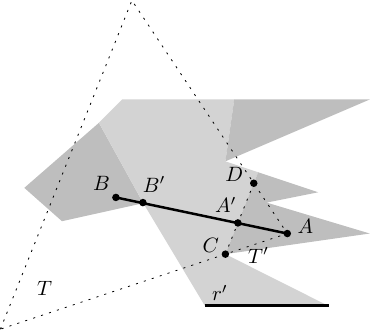}
\caption{Illustration for the proof of Claim~\ref{claim3} in the proof of Lemma \ref{lem:rooted-cover}.}
\label{fig:claim3}
\end{figure}

To put everything together, we set $\fT(A)\colonequals\bigcup_{i=1}^3\fT_i(A)$.
Then, it follows that $\bigcup_{i=1}^3\fV_i(A)\subseteq\fT(A)$ for every $A\in R$.
Clearly, the set $\fT(A)$ is measurable.
Summing the three estimates on areas of the trapezoids, we obtain
\begin{equation*}
\lambda_2(\fT(A))\leq\bigl(6\gamma^{-2}+3(1-\gamma)^{-2}\gamma^{-2}+4(1-\gamma)^{-2}-1\bigr)\smc(R)
\end{equation*}
for every point $A\in R$.
We choose $\gamma\in(0,1)$ so that the value of the coefficient is minimized.
For $x\in(0,1)$, the function $x\mapsto 6x^{-2}+3(1-x)^{-2}x^{-2}+4(1-x)^{-2}-1$ attains its minimum $86.7027<87$ at $x\approx 0.5186$.
Altogether, we have $\lambda_2(\fT(A))<87\smc(R)$ for every $A\in R$.

It remains to show that the set $\{(A,B)\colon A\in R$ and $B\in\fT(A)\}$ is measurable.
For every body $P$ of $R$ and for $i\in\{1,2,3\}$, the definition of the trapezoid $\fT_i(A)$ in Claim~$i$ implies that the set $\{(A,B)\colon A\in P$ and $B\in\fT_i(A)\}$ is the intersection of $P\times\bbR^2$ with a semialgebraic (hence measurable) subset of $(\bbR^2)^2$ and hence is measurable.
There are countably many bodies of $R$, as each of them has positive measure.
Therefore, $\{(A,B)\colon A\in R$ and $B\in\fT(A)\}$ is a countable union of measurable sets and hence is measurable.
\end{proof}

Let $S$ be a bounded open subset of the plane, and let $\ell$ be a diagonal of $S$ that lies on the $x$-axis.
For a point $A\in S$, we define the set
\begin{equation*}
\fS(A,\ell)\colonequals\{B\in\Vis(A,S)\colon AB\cap\ell\neq\emptyset\text{ and }|y(A)|\geq|y(B)|\}.
\end{equation*}
The following lemma is a slightly more general version of a result of Cabello et~al.\ \cite{CCK+14}.

\begin{lemma}
\label{lem:crossing-visibility}
Let\/ $S$ be a bounded open simply connected subset of\/ $\bbR^2$, and let\/ $\ell$ be its diagonal.
It follows that\/ $\lambda_2(\fS(A,\ell))\leq 3\smc(S)$ for every\/ $A\in S$.
\end{lemma}

\begin{proof}
We can assume without loss of generality that $\ell$ lies on the $x$-axis.
Using an argument similar to the proof of Lemma \ref{lem:open}, we can show that the set $\{B\in\Vis(A,S)\colon AB\cap\ell\neq\emptyset\}$ is open.
Therefore, $\fS(A,\ell)$ is the intersection of an open set and the closed half-plane $\{(x,y)\in\bbR^2\colon y\leq-y(A)\}$ or $\{(x,y)\in\bbR^2\colon y\geq-y(A)\}$, whichever contains~$A$.
Consequently, the set $\fS(A,\ell)$ is measurable for every $A\in S$.

We clearly have $\lambda_2(\fS(A,\ell))=0$ for points $A\in S\setminus\Vis(\ell,S)$.
By Lemma \ref{lem:visibility}, the set $\Vis(A,S)\cap\ell$ is an open subsegment $CD$ of~$\ell$.
The interior $\inte{T}$ of the triangle $T\colonequals ACD$ is contained in $S$.
Since $\inte{T}$ is a convex subset of $S$, we have $\lambda_2(\inte{T})=\frac{1}{2}|CD|\cdot|y(A)|\leq\smc(S)$.
Therefore, every point $B\in\fS(A,\ell)$ is contained in a trapezoid of height $|y(A)|$ with bases of length $|CD|$ and $2|CD|$.
The area of this trapezoid is $\frac{3}{2}|CD|\cdot|y(A)|\leq 3\smc(S)$.
Hence we have $\lambda_2(\fS(A,\ell))\leq 3\smc(S)$ for every point $A\in S$.
\end{proof}

\begin{proof}[Proof of Theorem \ref{thm:upper-bound-2d}]
In view of Lemma \ref{lem:reduction}, we can assume without loss of generality that $S$ is a bounded open simply connected set.
Let $\ell$ be a diagonal of~$S$.
We can assume without loss of generality that $\ell$ lies on the $x$-axis.

According to Lemma \ref{lem:cutting}, the set $S\setminus\ell$ has exactly two p-components $S_1$ and $S_2$, the sets $S_1\cup\ell$ and $S_2\cup\ell$ are rooted sets, and $\ell$ is their common root.
By Lemma~\ref{lem:rooted-cover}, for $i\in\{1,2\}$, every point $A\in S_i\cup\ell$ can be assigned a measurable set $\fT_i(A)$ so that $\lambda_2(\fT_i(A))<87\smc(S_i\cup\ell)\leq 87\smc(S)$, every line segment $\overline{BC}$ in $S_i\cup\ell$ satisfies $B\in\fT_i(C)$ or $C\in\fT_i(B)$, and the set $\{(A,B)\colon A\in S_i\cup\ell$ and $B\in\fT_i(A)\}$ is measurable.

We set $\fS(A)\colonequals\fT_i(A)\cup\fS(A,\ell)$ for every point $A\in S_i$ with $i\in\{1,2\}$.
We set $\fS(A)\colonequals\fT_1(A)\cup\fT_2(A)$ for every point $A\in\ell=S\setminus(S_1\cup S_2)$.
Let
\begin{equation*}
\fS\colonequals\{(A,B)\colon A\in S\text{ and }B\in\fS(A)\}\cup\{(B,A)\colon A\in S\text{ and }B\in\fS(A)\}\subseteq(\bbR^2)^2.
\end{equation*}
It follows that the set $\fS$ is measurable.

Let $\overline{AB}$ be a line segment in $S$, and suppose $|y(A)|\geq|y(B)|$.
Then either $A$ and $B$ are in distinct p-components of $S\setminus\ell$ or they both lie in the same component $S_i$ with $i\in\{1,2\}$.
In the first case, we have $B\in\fS(A)$, since $AB$ intersects $\ell$ and $\fS(A,\ell)\subseteq\fS(A)$.
In the second case, we have $B\in\fT_i(A)\subseteq\fS(A)$ or $A\in\fT_i(B)\subseteq\fS(B)$.
Therefore, we have $\Seg(S)\subseteq\fS$.
Since both $\Seg(S)$ and $\fS$ are measurable, we have
\begin{equation*}
\lambda_4(\Seg(S))\leq\lambda_4(\fS)\leq 2\int_{A\in S}\lambda_2(\fS(A)),
\end{equation*}
where the second inequality is implied by Fubini's Theorem.
The bound $\lambda_2(\fS(A))\leq 90\smc(S)$ implies
\begin{equation*}
\lambda_4(\Seg(S))\leq 2\int_S 90\smc(S)=180\smc(S)\lambda_2(S).
\end{equation*}
Finally, this bound can be rewritten as $\beer(S)=\lambda_4(\Seg(S))\lambda_2(S)^{-2}\leq 180\conv(S)$.
\end{proof}

\subsection{Proof of Lemma \ref{lem:reduction}}
\label{ssec:reduction}

In this section, we prove Lemma \ref{lem:reduction}, which reduces the general setting of Theorem~\ref{thm:upper-bound-2d} to the case that $S$ is a bounded open simply connected subset of~$\bbR^2$.

\begin{lemma}
\label{lem:bounded}
Let\/ $S\subseteq\bbR^2$ be a set whose\/ $\beer(S)$ is defined.
For every\/ $\varepsilon>0$, there is a bounded set\/ $S'\subseteq S$ such that\/ $\lambda(S')\geq(1-\varepsilon)\lambda(S)$ and\/ $\beer(S')\geq\beer(S)-\varepsilon$.
Moreover, if\/ $S$ is p-componentwise simply connected, then so is~$S'$.
\end{lemma}

\begin{proof}
Let $B$ be an open ball in $\bbR^2$ centered at the origin.
Consider the sets $S'=S\cap B$ and $S_0=S\setminus B$ partitioning the set $S$.
Fix the radius of $B$ large enough, so that $S_0$ has measure at most $\varepsilon\lambda(S)/2$.
We claim that $S'$ has the properties stated in the lemma.

Clearly $\lambda(S')\geq(1-\varepsilon/2)\lambda(S)>(1-\varepsilon)\lambda(S)$.
Moreover, $\Seg(S')=\Seg(S)\setminus((S_0\times S)\cup(S\times S_0))$, and hence $\Seg(S')$ is measurable and we have $\lambda_4(\Seg(S'))\geq\lambda_4(\Seg(S))-\varepsilon\lambda(S)^2$.
Therefore,
\begin{equation*}
\beer(S')=\frac{\lambda_4(\Seg(S'))}{\lambda_4(S'\times S')}\geq\frac{\lambda_4(\Seg(S'))}{\lambda_4(S\times S)}\geq\frac{\lambda_4(\Seg(S))-\varepsilon\lambda(S)^2}{\lambda_4(S\times S)}=\beer(S)-\varepsilon,
\end{equation*}
as claimed.
It is clear from the construction that if $S$ is p-componentwise simply connected, then so is~$S'$.
\end{proof}

\begin{lemma}
\label{lem:boundary}
Let\/ $S\subseteq\bbR^2$ be a bounded p-componentwise simply connected measurable set with measurable segment set.
Then\/ $\lambda_4(\Seg(S)\setminus\Seg(\inte{S}))=0$.
In other words, all the segments in\/ $S$ are in fact contained in\/ $\inte{S}$, except for a set of measure zero.
\end{lemma}

\begin{proof}
Let $\fB$ denote the set $\Seg(S)\setminus\Seg(\inte{S})$, that is, $\fB$ is the set of segments in $S$ containing at least one point of~$\bd S$.
Note that $\fB$ is measurable, since $\Seg(S)$ is measurable by assumption and $\Seg(\inte{S})$ is an open set by Lemma \ref{lem:open}, hence it is measurable as well.

Let $\overline{AB}$ be a segment contained in $S$, and let $C$ be a point of $\overline{AB}$.
We say that $C$ is an \emph{isolated boundary point} of the segment $\overline{AB}$, if $C\in\bd S$, but there is an $\varepsilon>0$ such that no other point of $\overline{AB}\cap\nei{\varepsilon}{C}$ belongs to~$\bd S$.

We partition the set $\fB$ into four parts as follows:
\begin{align*}
\fBv&\colonequals\{(A,B)\in\fB\colon A=B\text{ or }\overline{AB}\text{ is a vertical segment}\},\\
\fBl&\colonequals\{(A,B)\in\fB\setminus\fBv\colon A\text{ is an isolated boundary point of }\overline{AB}\},\\
\fBr&\colonequals\{(A,B)\in\fB\setminus(\fBv\cup\fBl)\colon B\text{ is an isolated boundary point of }\overline{AB}\},\\
\fBo&\colonequals\fB\setminus(\fBv\cup\fBl\cup\fBr).
\end{align*}
We claim that each of these sets has measure zero.
For $\fBv$, this is clear, since $\fBv$ is a subset of $\{(A,B)\in\bbR^2\times\bbR^2\colon A=B\text{ or }\overline{AB}\text{ is a vertical segment}\}$, which clearly has $\lambda_4$-measure zero.

Consider now the set $\fBl$. We first argue that it is measurable.
For a set $\alpha\subseteq [0,1]$ and a pair of points $(A,B)$, define $\overline{AB}[\alpha]\colonequals\{tB+(1-t)A\colon t\in\alpha\}$, and let $\fS(\alpha)$ be the set $\{(A,B)\in\bbR^2\times\bbR^2\colon\overline{AB}[\alpha]\subseteq\inte{S}\}$.
In particular, if $\alpha=[0,1]$ then $\overline{AB}[\alpha]=\overline{AB}$ and $\fS(\alpha)=\Seg(\inte{S})$.
If $\alpha$ is a closed interval, then $\overline{AB}[\alpha]$ is a segment, and it is not hard to see that $\fS(\alpha)$ is an open set, and, in particular, it is measurable.
If $\alpha$ is an open interval, say $\alpha=(s,t)\subseteq[0,1]$, then $\fS(\alpha)=\bigcap_{n\in\bbN}\fS([s+n^{-1},t-n^{-1}])$, and hence $\fS(\alpha)$ is measurable as well.
We then see that
\begin{equation*}
\fBl=\fB\cap(\bd{S}\times S)\cap\Bigl(\bigcup_{n\in\bbN}\fS((0,n^{-1}))\Bigr),
\end{equation*}
showing that $\fBl$ is measurable.
An analogous argument shows that $\fBr$ is measurable, and hence $\fBo$ is measurable as well.

In the rest of the proof, we will use two basic facts of integral calculus, which we now state explicitly.

\begin{fact}[{see \cite[Lemma 7.25 and Theorem 7.26]{Rud87}}]
\label{fact1}
Let\/ $X,Y\subseteq\bbR^d$ be two open sets, and let\/ $\sigma\colon X\to Y$ be a bijection such that both\/ $\sigma$ and\/ $\sigma^{-1}$ are continuous and differentiable on\/ $X$ and\/ $Y$, respectively.
Then, for any\/ $X_0\subseteq X$, the set\/ $X_0$ is measurable if and only if\/ $\sigma(X_0)$ is measurable.
Moreover, $\lambda(X_0)=0$ if and only if\/ $\lambda(\sigma(X_0))=0$.
\end{fact}

\begin{fact}[Fubini's Theorem, see {\cite[Theorem 8.12]{Rud87}}]
Let\/ $M\subseteq\bbR^k\times\bbR^\ell$ be a measurable set.
For\/ $x\in\bbR^k$, define\/ $M_x\colonequals\{y\in\bbR^\ell\colon(x,y)\in M\}$.
Then, for almost every\/ $x\in\bbR^k$, the set\/ $M_x$ is\/ $\lambda_\ell$-measurable, and
\begin{equation*}
\lambda_{k+\ell}(M)=\int_{x\in\bbR^k}\lambda_\ell(M_x).
\end{equation*}
\end{fact}

Let us prove that $\lambda_4(\fBl)=0$.
The basic idea is as follows: suppose that we have fixed a non-vertical line $L$ and a point $B\in L$.
It can be easily seen that there are at most countably many points $A\in L$ such that $(A,B)\in\fBl$.
Since a line $L$ with a point $B\in L$ can be determined by three parameters, we will see that $\fBl$ has $\lambda_4$-measure zero.

Let us describe this reasoning more rigorously. Let $L_{a,b}$ denote the line $\{(x,y)\in\bbR^2\colon y=ax+b\}$.
Define a mapping $\sigma\colon\bbR^4\to\bbR^2\times\bbR^2$ as follows: $\sigma(a,b,x,x')=(A,B)$, where $A=(x,ax+b)$ and $B=(x',ax'+b)$.
In other words, $\sigma(a,b,x,x')$ is the pair of points on the line $L_{a,b}$ whose horizontal coordinates are $x$ and $x'$, respectively.
For every non-vertical segment $\overline{AB}$, there is a unique quadruple $(a,b,x,x')$ with $x\neq x'$, such that $\sigma(a,b,x,x')=(A,B)$.
In particular, $\sigma$ is a bijection from the set $\{(a,b,x,x')\in\bbR^4\colon x\neq x'\}$ to the set $\{(A,B)\in\bbR^2\times\bbR^2\colon A,B$ not on the same vertical line$\}$.

Define $\widehat\fBl=\sigma^{-1}(\fBl)$.
Note that $\sigma$ satisfies the assumptions of Fact~\ref{fact1}, and therefore $\widehat\fBl$ is measurable.
Moreover, $\lambda_4(\widehat\fBl)=0$ if and only if $\lambda_4(\fBl)=0$.

For a fixed triple $(a,b,x')\in\bbR^3$, let $X_{a,b,x'}$ denote the set $\{x\in\bbR\colon(a,b,x,x')\in\widehat\fBl\}$.
We claim that $X_{a,b,x'}$ is countable.
To see this, choose a point $x\in X_{a,b,x'}$ and define $(A,B)\colonequals\sigma(a,b,x,x')$.
Since $(A,B)\in\fBl$, we know that $A$ is an isolated boundary point of $\overline{AB}$, which implies that there is a closed interval $\beta\subseteq\bbR$ of positive length such that $\beta\cap X_{a,b,x'}=\{x\}$.
This implies that $X_{a,b,x'}$ is countable and thus of measure zero.

Since $\widehat\fBl$ is measurable, we can apply Fubini's Theorem to get
\begin{equation*}
\lambda_4(\widehat\fBl)=\int_{(a,b,x')\in\bbR^3}\lambda_1(X_{a,b,x'}).
\end{equation*}
Therefore $\lambda_4(\widehat\fBl)=0$ as claimed.
A similar argument shows that $\lambda_4(\widehat\fBr)=0$.

It remains to deal with the set~$\fBo$.
We will use the following strategy: we will fix two parallel non-horizontal lines $L_1,L_2$, and study the segments orthogonal to these two lines, with one endpoint on $L_1$ and the other on~$L_2$.
Roughly speaking, our goal is to show that for ``almost every'' choice of $L_1$ and $L_2$, there are ``almost no'' segments of this form belonging to~$\fBo$.

Let $L'_{a,b}$ denote the (non-horizontal) line $\{(ay+b,y)\colon y\in\bbR\}$.
Let us say that a pair of distinct points $(A,B)$ has \emph{type $(a,b,c)$}, if $A\in L'_{a,b}$, $B\in L'_{a,c}$, and the segment $\overline{AB}$ is orthogonal to $L'_{a,b}$ (and therefore also to $L'_{a,c}$).
The value $a$ is then called \emph{the slope} of the type $t=(a,b,c)$.

Note that every pair of distinct points $(A,B)$ defining a non-vertical segment has a unique type $(a,b,c)$, with $b\neq c$.
Define a mapping $\tau\colon\bbR^4\to\bbR^2\times\bbR^2$, where $\tau(a,b,c,y)$ is the pair of points $(A,B)$ of type $(a,b,c)$ such that $A=(ay+b,y)$.
Note that $\tau$ is a bijection from the set $\{(a,b,c,y)\in\bbR^4\colon b\neq c\}$ to the set $\{(A,B)\in\bbR^2\times\bbR^2\colon A,B$ not on the same vertical line$\}$.
We can easily verify that $\tau$ satisfies the assumptions of Fact~\ref{fact1}.

Define $\widetilde\fBo=\tau^{-1}(\fBo)$.
From Fact~\ref{fact1}, it follows that $\widetilde\fBo$ is measurable, and $\lambda_4(\fBo)=0$ if and only if $\lambda_4(\widetilde\fBo)=0$.
For a type $t=(a,b,c)\in\bbR^3$, define $Y_t=\{y\in\bbR\colon(a,b,c,y)\in\widetilde\fBo\}$.
Furthermore, for a set $\alpha\subseteq[0,1]$, define $\fBo(\alpha)=\fBo\cap\fS(\alpha)$, $\widetilde\fBo(\alpha)=\tau^{-1}(\fBo(\alpha))$, and $Y_t(\alpha)=\{y\in\bbR\colon(a,b,c,y)\in\widetilde\fBo(\alpha)\}$.
In our applications, $\alpha$ will always be an interval (in fact, an open interval with rational endpoints), and in such case we already know that $\fBo(\alpha)$ is measurable, hence $\widetilde\fBo(\alpha)$ is measurable.

By Fubini's Theorem, we have
\begin{equation}
\label{equation}
\lambda_4(\widetilde\fBo)=\int_{a\in\bbR}\int_{(b,c)\in\bbR^2}\lambda_1(Y_{(a,b,c)}),\tag{$*$}
\end{equation}
and $Y_t$ is measurable for all $t\in\bbR^3$ up to a set of $\lambda_3$-measure zero.
An analogous formula holds for $\widetilde\fBo(\alpha)$ and $Y_t(\alpha)$ for any open interval $\alpha\subseteq[0,1]$ with rational endpoints.
Since there are only countably many such intervals, and a countable union of sets of measure zero has measure zero, we know that there is a set $T_0\subseteq\bbR^3$ of measure zero, such that for all $t\in\bbR^3\setminus T_0$ the set $Y_t$ is measurable, and moreover for any rational interval $\alpha$ the set $Y_t(\alpha)$ is measurable as well.

Our goal is to show that there are at most countably many slopes $a\in\bbR$ for which there is a $(b,c)\in\bbR^2$ such that $\lambda_1(Y_{(a,b,c)})>0$.
From \eqref{equation} it will then follow that $\widetilde\lambda_4(\fBo)=0$.
To achieve this goal, we will show that to any type $t$ for which $\lambda_1(Y_t)>0$, we can assign a set $R_t\subseteq\bd S$ of positive $\lambda_2$-measure (the \emph{region} of $t$), so that if $t$ and $t'$ have different slopes and if $Y_{t}$ and $Y_{t'}$ both have positive measure, then $R_t$ and $R_{t'}$ are disjoint.
Since there cannot be uncountably many disjoint sets of positive measure, this will imply the result.

Let us fix a type $t=(a,b,c)\in\bbR^3\setminus T_0$ such that $\lambda_1(Y_t)>0$.
Let us say that an element $y\in Y_t$ is \emph{half-isolated} if there is an $\varepsilon>0$ such that $[y,y+\varepsilon]\cap Y_t=\{y\}$ or $[y-\varepsilon,y]\cap Y_t=\{y\}$.
Clearly, $Y_t$ has at most countably many half-isolated elements.
Define $Y_t^*\colonequals\{y\in Y_t\colon y$ is not half-isolated$\}$.
Of course, $\lambda_1(Y_t^*)=\lambda_1(Y_t)$.
See Figure~\ref{fig:boundary} for an illustration.

\begin{figure}[t]
\centering
\includegraphics{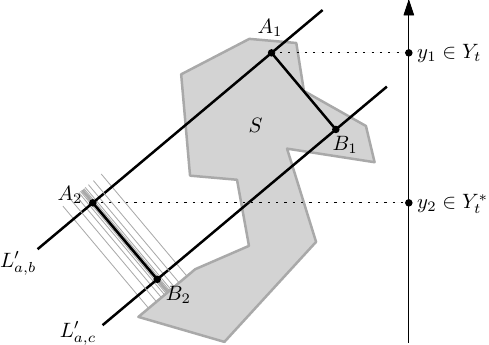}
\caption{An illustration for the proof of Lemma \ref{lem:boundary}.
The element $y_1$ of $Y_t$ is half-isolated while $y_2$ is not.}
\label{fig:boundary}
\end{figure}

Choose $y\in Y_t^*$, and define $(A_y,B_y)\colonequals\tau(a,b,c,y)$.
We claim that $\overline{A_yB_y}\cap\inte{S}$ is either empty or a single interval.
Let us choose any two points $C,D\in\overline{A_yB_y}\cap\inte{S}$.
We will show that the segment $\overline{CD}$ is inside~$\inte{S}$.
For $\varepsilon>0$ small enough, the neighborhoods $\nei{\varepsilon}{C}$ and $\nei{\varepsilon}{D}$ are subsets of~$S$.
Since $y$ is not half-isolated in $Y_t$, we can find two segments $P,Q\in\fBo$ of type $t$ that intersect both $\nei{\varepsilon}{C}$ and $\nei{\varepsilon}{D}$, with $\overline{A_yB_y}$ being between $P$ and~$Q$.
We can then find a closed polygonal curve $\Gamma\subseteq P\cup Q\cup\nei{\varepsilon}{C}\cup\nei{\varepsilon}{D}$ whose interior region contains $\overline{CD}$.
Since $S$ is p-componentwise simply connected, we see that $\overline{CD}\subseteq\inte{S}$.
Therefore, $\overline{A_yB_y}\cap\inte{S}$ is indeed an interval.

Since $\bd S$ is a closed set, we know that for every $y\in Y_t^*$, the set $\overline{A_yB_y}\cap\bd S$ is closed as well.
Moreover, neither $A_y$ nor $B_y$ are isolated boundary points of $\overline{A_yB_y}$, because then $(A_y,B_y)$ would belong to $\fBl$ or~$\fBr$.
We conclude that $\overline{A_yB_y}\cap\bd S$ is either equal to a single closed segment of positive length containing $A_y$ or $B_y$, or it is equal to a disjoint union of two closed segments of positive length, one of which contains $A_y$ and the other contains~$B_y$.

For an integer $n\in\bbN$, define two sets $Y_t^\lhd(n)$ and $Y_t^\rhd(n)$ by
\begin{align*}
Y_t^\lhd(n)&\colonequals\{y\in Y_t^*\colon\overline{A_yB_y}[(0,n^{-1})]\subseteq\bd S\},\\
Y_t^\rhd(n)&\colonequals\{y\in Y_t^*\colon\overline{A_yB_y}[(1-n^{-1},1)]\subseteq\bd S\}.
\end{align*}
Note that these sets are measurable: for instance, $Y_t^\lhd(n)$ is equal to $Y_t^*\setminus\left(\bigcup_\alpha Y_t(\alpha)\right)$, where we take the union over all rational intervals $\alpha$ intersecting $(0,n^{-1})$.
Moreover, we have $Y_t^*=\bigcup_{n\in\bbN}(Y_t^\lhd(n)\cup Y_t^\rhd(n))$.
It follows that there is an $n$ such that $Y_t^\lhd(n)$ or $Y_t^\rhd(n)$ has positive measure.
Fix such an $n$ and assume, without loss of generality, that $\lambda_1(Y_t^\lhd(n))$ is positive.
Define \emph{the region of $t$}, denoted by $R_t$, by
\begin{equation*}
R_t\colonequals\bigcup_{y\in Y_t^\lhd(n)}\overline{A_yB_y}[(0,n^{-1})].
\end{equation*}
The set $R_t$ is a bijective affine image of $Y_t^\lhd(n)\times(0,n^{-1})$, and in particular it is $\lambda_2$-measurable with positive measure.
Note that $R_t$ is a subset of~$\bd S$.

Consider now two types $t,t'\in\bbR^3\setminus T_0$ with distinct slopes, such that both $Y_t$ and $Y_{t'}$ have positive measure.
We will show that the regions $R_t$ and $R_{t'}$ are disjoint.

For contradiction, suppose there is a point $C\in R_t\cap R_{t'}$.
Let $\overline{AB}$ and $\overline{A'B'}$ be the segments containing $C$ and having types $t$ and $t'$, respectively.
Fix $\varepsilon>0$ small enough, so that none of the four endpoints $A,B,A',B'$ lies in $\nei{\varepsilon}{C}$.
Since $Y_t^*$ has no half-isolated points of $Y_t$, we know that $\fBo$ has segments of type $t$ arbitrarily close to $\overline{AB}$ on both sides of $\overline{AB}$, and similarly for segments of type $t'$ close to $\overline{A'B'}$.
We can therefore find four segments $P,Q,P',Q'\in\fBo\setminus\{\overline{AB},\overline{A'B'}\}$ with these properties:
\begin{itemize}
\item $P$ and $Q$ have type $t$, and $P'$ and $Q'$ have type~$t'$.
\item $\overline{AB}$ is between $P$ and $Q$ (i.e., $\overline{AB}\subseteq\Conv(P\cup Q)$) and $\overline{A'B'}$ is between $P'$ and~$Q'$.
\item Both $P$ and $Q$ intersect both $P'$ and $Q'$ inside $\nei{\varepsilon}{C}$.
\end{itemize}

We see that the four points where $P\cup Q$ intersects $P'\cup Q'$ form the vertex set of a parallelogram $W$ whose interior contains the point~$C$.
Moreover, the boundary of $W$ is a closed polygonal curve contained in~$S$.
Since $S$ is p-componentwise simply connected, $W$ is a subset of $S$ and $C$ belongs to~$\inte{S}$.
This is a contradiction, since all points of $R_t$ (and $R_{t'}$) belong to~$\bd S$.

We conclude that $R_t$ and $R_{t'}$ are indeed disjoint.
Since there cannot be uncountably many disjoint sets of positive measure in $\bbR^2$, there are at most countably many values $a\in\bbR$ for which there is a type $t=(a,b,c)$ with $\lambda_1(Y_t)$ positive.
Consequently, the right-hand side of \eqref{equation} is zero, and so $\lambda_4(\fBo)=0$, as claimed.
\end{proof}

\begin{proof}[Proof of Lemma \ref{lem:reduction}]
Observe that the inequalities $\beer(S)\leq\alpha\conv(S)$ and $\lambda_4(\Seg(S))\leq\alpha\smc(S)\cdot\lambda_2(S)$ are equivalent.
Call a set $S$ \emph{bad} if $\Seg(S)$ is measurable and $\beer(S)>\alpha\conv(S)$ or equivalently $\lambda_4(\Seg(S))>\alpha\smc(S)\lambda_2(S)$.
To prove the lemma, we suppose for the sake of contradiction that there exists a bad p-componentwise simply connected set $S\subseteq\bbR^2$ of finite positive measure.

By Lemma \ref{lem:bounded}, for each $\varepsilon>0$, there is a bounded p-componentwise simply connected set $S'\subseteq S$ such that $\lambda_2(S')\geq(1-\varepsilon)\lambda_2(S)$ and $\beer(S')\geq\beer(S)-\varepsilon$.
In particular, such a set $S'$ satisfies $\conv(S')\leq\conv(S)/(1-\varepsilon)$.
Hence, for $\varepsilon$ small enough, the set $S'$ is bad.

Let $S''$ be the interior of~$S'$.
By Lemma \ref{lem:boundary}, $\lambda_4(\Seg(S''))=\lambda_4(\Seg(S'))$.
Clearly, $\lambda_2(S'')\leq\lambda_2(S')$ and $\smc(S'')\leq\smc(S')$, and therefore $S''$ is bad as well.

Note that $S''$ is p-componentwise simply connected.
Since $S''$ is an open set, all its p-components are open as well.
In particular, $S''$ has at most countably many p-components.
Let $\calC$ be the set of p-components of~$S''$.
Each $T\in\calC$ is a bounded open simply connected set, and therefore cannot be bad.
Therefore,
\begin{equation*}
\lambda_4(\Seg(S''))=\sum_{T\in\calC}\lambda_4(\Seg(T))\leq\sum_{T\in\calC}\alpha\smc(T)\lambda_2(T)\leq\alpha\smc(S'')\lambda_2(S''),
\end{equation*}
showing that $S''$ is not bad.
This is a contradiction.
\end{proof}

\subsection{Proof of Lemma \ref{lem:cutting}}
\label{ssec:cutting}

Here we prove Lemma \ref{lem:cutting}, which says that every bounded open simply connected subset of $\bbR^2$ can can be split by a diagonal into two rooted sets.

\begin{lemma}
\label{lem:cutting-aux}
Let\/ $S$ be a bounded open simply connected subset of\/ $\bbR^2$, and let\/ $\ell$ be a diagonal of\/~$S$.
Let\/ $h^-$ and\/ $h^+$ be the open half-planes defined by the supporting line of\/~$\ell$.
It follows that the set\/ $S\setminus\ell$ has exactly two p-components\/ $S_1$ and\/~$S_2$.
Moreover, for every point\/ $A\in\ell$ and every neighborhood\/ $\nei{\varepsilon}{A}\subseteq S$, we have\/ $\nei{\varepsilon}{A}\cap h^-\subseteq S_1$ and\/ $\nei{\varepsilon}{A}\cap h^+\subseteq S_2$.
\end{lemma}

\begin{proof} Notice first that any p-component of an open set is also open.
This implies that any path-connected open set is also p-connected, and therefore every open simply connected set is p-connected as well.

Let $A\in\ell$, and let $\nei{\varepsilon}{A}$ be a neighborhood of $A$ contained in $S$.
We choose arbitrary points $B\in\nei{\varepsilon}{A}\cap h^-$ and $C\in\nei{\varepsilon}{A}\cap h^+$.
Suppose for a contradiction that $S\setminus\ell$ has a single p-component.
Then there exists a polygonal curve $\Gamma$ in $S\setminus\ell$ with endpoints $B$ and~$C$.
Let $\Delta\subseteq S$ be the closed polygonal curve $\Gamma\cup\overline{BC}$.
We can assume that the curve $\Delta$ is simple using a local redrawing argument.
See Figure~\ref{fig:redrawing}.

\begin{figure}[t]
\centering
\includegraphics{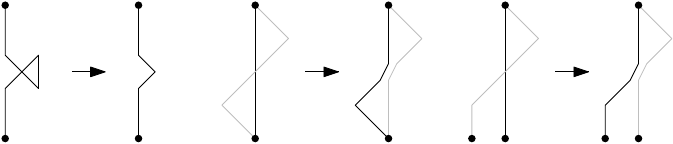}
\caption{Removing self-intersections and intersections between adjacent polygonal lines.}
\label{fig:redrawing}
\end{figure}

The curve $\Delta$ separates $\bbR^2$ into two regions.
The closure $\overline{\ell}$ of the diagonal $\ell$ is a closed line segment that intersects $\Delta$ in exactly one point.
It follows that one endpoint of $\ell$ is in the interior region of~$\Delta$.
Since the endpoints of $\ell$ do not belong to $S$, this contradicts the assumption that $S$ is simply connected.

Now, we show that the set $S\setminus\ell$ has at most two p-components.
For a point $D\in\ell$, let $\nei{\varepsilon}{D}$ be a neighborhood of $D$ in~$S$.
The set $\nei{\varepsilon}{D}\cap h^-$ is contained in a unique p-component $S_1$ of $S\setminus\ell$, and $\nei{\varepsilon}{D}\cap h^+$ is contained in a different p-component~$S_2$.
Choose another point $E\in\ell$ with a neighborhood $\nei{\varepsilon'}{E}\subseteq S$.
We claim that $\nei{\varepsilon'}{E}\cap h^-$ also belongs to~$S_1$.
To see this, note that since $\overline{DE}$ is a compact subset of the open set $S$, it has a neighborhood $\nei{\delta}{\overline{DE}}$ which is contained in~$S$.
Clearly, $\nei{\delta}{\overline{DE}}\cap h^-$ is p-connected and therefore belongs to~$S_1$, hence $\nei{\varepsilon'}{E}\cap h^-$ belongs to $S_1$ as well.
An analogous argument can be made for the half-plane $h^+$ and the p-component~$S_2$.

Since for every p-component $S'$ of $S\setminus\ell$, there is a point $A\in\ell$ and a neighborhood $\nei{\varepsilon}{A}\subseteq S$ such that $\nei{\varepsilon}{A}\cap S'\neq\emptyset$, we see that $S_1$ and $S_2$ are the only two p-components of $S\setminus\ell$.
\end{proof}

\begin{proof}[Proof of Lemma \ref{lem:cutting}]
By Lemma \ref{lem:cutting-aux}, the set $S\setminus\ell$ has of exactly two p-components $S_1$ and~$S_2$.
It remains to show that $S_1\cup\ell$ and $S_2\cup\ell$ are rooted sets.

Since $S_1$ and $S_2$ are p-connected, $S_1\cup\ell$ and $S_2\cup\ell$ are p-connected as well.
To show that $S_1\cup\ell$ and $S_2\cup\ell$ are simply connected, choose a Jordan curve $\Gamma$ in, say, $S_1\cup\ell$, and let $Z$ be the interior region of~$\Gamma$.
Suppose for a contradiction that $Z$ is not a subset of $S_1\cup\ell$.
Since $S$ is simply connected, we have $Z\subseteq S$.
Hence there is a point $A\in Z\cap S_2$.
Since both $S_2$ and $Z$ are open, we can assume that $A$ does not lie on the supporting line of~$\ell$.
Let $\overline{AB}$ be the minimal closed segment parallel to $\ell$ such that $B\in\Gamma$.
Then $B$ belongs to $S_1$, $A$ belongs to $S_2$, and yet $A$ and $B$ are in the same p-component of $S\setminus\ell$.
This contradiction shows that $S_1\cup\ell$ and $S_2\cup\ell$ are simply connected.

As subsets of the bounded set $S$, the sets $S_1\cup\ell$ and $S_2\cup\ell$ are bounded.
Lemma~\ref{lem:cutting-aux} and the fact that $S_i$ is open imply that the set $S_i\cup\ell$ is half-open and $S_i\cap\bd S_i=\ell$ for $i\in\{1,2\}$.
Therefore, the sets $S_1\cup\ell$ and $S_2\cup\ell$ are rooted, and $\ell$ is their root.
\end{proof}

\subsection{Proof of Lemma \ref{lem:branch-structure}}
\label{ssec:branch-structure}

Here we prove Lemma \ref{lem:branch-structure}, which explains the tree structure of rooted sets.
For this entire section, let $R$ be a rooted set and $(L_k)_{k\geq 1}$ be the partition of $R$ into levels.
We will need several auxiliary results in order to prove Lemma \ref{lem:branch-structure}.

For disjoint sets $S,T\subseteq\bbR^2$, we say that the set $S$ is \emph{$T$-half-open} if every point $A\in S$ has a neighborhood $\nei{\varepsilon}{A}$ that satisfies one of the following two conditions:
\begin{enumerate}
\item\label{item:T-half-open-1} $\nei{\varepsilon}{A}\subseteq S$,
\item\label{item:T-half-open-2} $\nei{\varepsilon}{A}\cap\bd S$ is a diameter of $\nei{\varepsilon}{A}$ splitting it into two subsets, one of which (including the diameter) is $\nei{\varepsilon}{A}\cap S$ and the other (excluding the diameter) is $\nei{\varepsilon}{A}\cap T$.
\end{enumerate}
The only difference with the definition of $S$ being half-open is that we additionally specify the ``other side'' of the neighborhoods $\nei{\varepsilon}{A}$ for points $A\in S\cap\bd S$ in the condition~\ref{item:T-half-open-2}.
A rooted set $R$ is $T$-half-open if and only if it is attached to $T$ according to the definition of attachment from Section~\ref{sec:visibility-2d}.

\begin{lemma}
\label{lem:L_1-half-open}
The set\/ $L_1$ is\/ $(\bbR^2\setminus R)$-half-open and\/ $L_1\cap\bd L_1=R\cap\bd R$.
\end{lemma}

\begin{proof}
We consider two cases for a point $A\in L_1$.
First, suppose $A\in L_1\cap\bd R$.
It follows that $A$ has a neighborhood $\nei{\varepsilon}{A}$ that satisfies the condition~\ref{item:half-open-2} of the definition of a half-open set.
By the definition of $L_1$, the same neighborhood $\nei{\varepsilon}{A}$ satisfies the condition~\ref{item:T-half-open-2} for $L_1$ being an $(\bbR^2\setminus R)$-half-open set.
In particular, $A\in\bd L_1$.
Since $R\cap\bd R\subseteq L_1$ by the definition of $L_1$, we have $R\cap\bd R\subseteq L_1\cap\bd L_1$.

Now, suppose $A\in L_1\cap\inte{R}$.
Let $B$ be a point of the root of $R$ such that $\overline{AB}\subseteq R$.
We have $\overline{AB}\setminus\{B\}\subseteq\inte{R}$, as otherwise the point $t'A+(1-t')B$ for $t'\colonequals\sup\{t\in[0,1]\colon At+(1-t)B\in\overline{AB}\cap\bd R\}$ would contradict the fact that $R$ is half-open.
There is a family of neighborhoods $\{\nei{\varepsilon_C}{C}\}_{C\in\overline{AB}}$ such that all $\nei{\varepsilon_D}{D}$ with $D\in\overline{AB}\setminus\{B\}$ satisfy the condition~\ref{item:half-open-1} and $\nei{\varepsilon_B}{B}$ satisfies the condition~\ref{item:half-open-2} for $R$ being half-open.
Since $\overline{AB}$ is compact, there is a finite set $X\subseteq\overline{AB}$ such that $\overline{AB}\subseteq\bigcup_{C\in X}\nei{\varepsilon_C/2}{C}$.
Hence $\nei{\varepsilon}{\overline{AB}}\subseteq\bigcup_{C\in X}\nei{\varepsilon_C}{C}$, where $\varepsilon\colonequals\min_{C\in X}\varepsilon_C/2$.
It follows that $\nei{\varepsilon}{\overline{AB}}\cap\bd R$ is an open segment $Q$ containing $B$ but not $A$ and splitting $\nei{\varepsilon}{\overline{AB}}$ into two subsets, one of which (including $Q$) is $\nei{\varepsilon}{\overline{AB}}\cap R$ and the other (excluding $Q$) is $\nei{\varepsilon}{\overline{AB}}\setminus R$. Let $\varepsilon'$ be the minimum of $\varepsilon$ and the distance of $A$ to the line containing~$Q$.
It follows that $\nei{\varepsilon'}{A}\subseteq\nei{\varepsilon}{\overline{AB}}\cap R$.
Therefore, for every $A'\in\nei{\varepsilon'}{A}$, we have $\overline{A'B}\subseteq R$, hence $\nei{\varepsilon'}{A}\subseteq L_1$.
It also follows that $L_1\cap\bd L_1\subseteq R\cap\bd R$.
\end{proof}

We say that a set $P\subseteq R$ is \emph{$R$-convex} when the following holds for any two points $A,B\in P$: if $\overline{AB}\subseteq R$, then $\overline{AB}\subseteq P$.

\begin{lemma}
\label{lem:L_1-R-convex}
The set\/ $L_1$ is\/ $R$-convex.
\end{lemma}

\begin{proof}
This follows directly from Lemma \ref{lem:visibility}.
\end{proof}

A \emph{branch} of $R$ is a p-component of $\bigcup_{k\geq 2}L_k$.

\begin{lemma}
\label{lem:branch-R-convex}
Every branch of\/ $R$ is\/ $R$-convex.
\end{lemma}

\begin{proof}
Let $P$ be a branch of $R$, and let $A,B\in P$ be such that $\overline{AB}\subseteq R$.
Since $R$ is half-open, it follows that $\overline{AB}\subseteq\inte{R}$.
Suppose $\overline{AB}\not\subseteq P$.
It follows that $\overline{AB}\cap L_1\neq\emptyset$.
Since $L_1$ is $(\bbR^2\setminus R)$-half-open (Lemma \ref{lem:L_1-half-open}) and $R$-convex (Lemma \ref{lem:L_1-R-convex}), we see that $\overline{AB}\cap L_1$ is an open segment $A'B'$ for some $A',B'\in\overline{AB}$.
It follows that $A',B'\in P$.

There is a simple polygonal line in $P$ connecting $A'$ with $B'$, which together with $\overline{A'B'}$ forms a Jordan curve $\Gamma$ in~$R$.
Now, let $C\in A'B'$.
Since $C\in L_1$, there is a point $D$ on the root of $R$ such that $\overline{CD}\subseteq R$.
Since $A',B'\notin L_1$, $D$ does not lie on the supporting line of $A'B'$.
Extend the segment $DC$ beyond $C$ until hitting $\bd R$ at a point~$C'$.
Here we use the fact that $R$ is bounded.
Since $R$ is simply connected, the entire interior region of $\Gamma$ is contained in $R$, so the points $D$ and $C'$ both lie in the exterior region of~$\Gamma$.
However, since $\Gamma\cap L_1=A'B'$, the line segment $\overline{DC'}$ crosses $\Gamma$ at exactly one point, which is~$C$.
This is a contradiction.
\end{proof}

\begin{lemma}
\label{lem:branch-simply-connected}
The set\/ $L_1$ and every branch of\/ $R$ are p-connected and simply connected.
\end{lemma}

\begin{proof}
Let $P$ be the set $L_1$ or a branch of~$R$.
It follows directly from the definitions of $L_1$ and a branch of $R$ that $P$ is p-connected.
To see that $P$ is simply connected, let $\Gamma$ be a Jordan curve in $P$, $A$ be a point in the interior region of $\Gamma$, and $BC$ be an inclusion-maximal open line segment in the interior region of $\Gamma$ such that $A\in BC$.
It follows that $B,C\in\Gamma$ and $\overline{BC}\subseteq R$, as $R$ is simply connected.
Since $B,C\in P$ and $P$ is $R$-convex (Lemmas \ref{lem:L_1-R-convex} and \ref{lem:branch-R-convex}), we have $A\in P$.
\end{proof}

\begin{lemma}
\label{lem:branch-half-open}
Every branch of\/ $R$ is\/ $L_1$-half-open.
\end{lemma}

\begin{proof}
Let $P$ be a branch of~$R$.
It is enough to check the condition~\ref{item:T-half-open-2} for $P$ being $L_1$-half-open for points in $\bd P\cap P$.
Let $A\in\bd P\cap P$.
Since $R$ is half-open, $A$ has a neighborhood $\nei{\varepsilon}{A}$ that satisfies the condition \ref{item:half-open-1} or \ref{item:half-open-2} for $S$ being half-open.
It cannot be \ref{item:half-open-2}, as then $A$ would lie on the root of $R$ and thus in~$L_1$.
Hence $\nei{\varepsilon}{A}\subseteq R$.

Since $L_1$ is $(\bbR^2\setminus R)$-half-open (Lemma \ref{lem:L_1-half-open}) and $R$-convex (Lemma \ref{lem:L_1-R-convex}) and $A\notin L_1$, the set $\nei{\varepsilon}{A}\cap L_1$ lies entirely in some open half-plane $h$ whose boundary line passes through~$A$.
The set $\nei{\varepsilon}{A}\setminus h$ is p-connected and contains $A$, so it lies entirely within~$P$.
The set $\nei{\varepsilon}{A}\cap h$ is disjoint from $P$.
Indeed, if there was a point $B\in\nei{\varepsilon}{A}\cap h\cap P$, then by the $R$-convexity of $P$ (Lemma \ref{lem:branch-R-convex}), the convex hull of $\nei{\varepsilon}{A}\setminus h$ and $B$ would lie entirely within $P$ and would contain $A$ in its interior, which would contradict the assumption that $A\in\bd P$.
It follows that $\nei{\varepsilon}{A}\cap\bd P$ is an open segment that partitions $\nei{\varepsilon}{A}$ into two half-discs, one of which (including $\nei{\varepsilon}{A}\cap\bd P$) is $\nei{\varepsilon}{A}\cap P$.

We show that $\nei{\varepsilon}{A}\cap\bd P\subseteq\bd L_1$.
Suppose to the contrary that there is a point $A'\in\nei{\varepsilon}{A}\cap\bd P\setminus\overline{L_1}$.
It follows that $A'$ has a neighborhood $\nei{\varepsilon'}{A'}\subseteq\nei{\varepsilon}{A}\setminus\overline{L_1}$.
Since $\nei{\varepsilon'}{A'}$ is p-connected and contains a point of $P$, it lies entirely within $P$. This contradicts the assumption that $A'\in\bd P$.

Since $\nei{\varepsilon}{A}\cap\bd P\subseteq\bd L_1$, there is a point $B\in\nei{\varepsilon}{A}\cap L_1$. Let $A'\in\nei{\varepsilon}{A}\cap\bd P$.
Since $L_1$ is $(\bbR^2\setminus R)$-half-open and $R$-convex and $A\notin L_1$, there is a point $C\in\overline{A'B}$ such that $\overline{CB}\setminus\{C\}\subseteq L_1$ while $\overline{A'C}$ is disjoint from~$L_1$.
The latter implies that $\overline{A'C}\subseteq P$, as $A'\in P$.
Hence $C=A'$.
This shows the whole triangle $T$ spanned by $\nei{\varepsilon}{A}\cap\bd P$ and $B$ excluding the open segment $\nei{\varepsilon}{A}\cap\bd P$ is contained in~$L_1$.

Since $A$ lies in the interior of $\nei{\varepsilon}{A}\cap(P\cup T)$, it has a neighborhood $\nei{\varepsilon'}{A}$ that lies entirely within $\nei{\varepsilon}{A}\cap(P\cup T)$.
This neighborhood witnesses the condition~\ref{item:T-half-open-2} for $P$ being $L_1$-half-open.
\end{proof}

\begin{lemma}
\label{lem:branch-bd-segment-in-S}
Let\/ $P$ be a branch of\/~$R$.
If\/ $A_0,A_1\in P\cap\bd P$, then\/ $\overline{A_0A_1}\subseteq R$.
\end{lemma}

\begin{proof}
Let $A_0,A_1\in P\cap\bd P$.
By Lemma \ref{lem:branch-half-open}, $P$ is $L_1$-half-open, hence there are $B_0,B_1\in L_1$ such that $\overline{A_0B_0}\setminus\{A_0\}\subseteq L_1$ and $\overline{A_1B_1}\setminus\{A_1\}\subseteq L_1$.
There is a polygonal line $\Gamma_1$ in $P$ connecting $A_0$ with $A_1$, and a polygonal line $\Gamma_2$ in $L_1$ connecting $B_0$ with~$B_1$.
These polygonal lines together with the line segments $\overline{A_0B_0}$ and $\overline{A_1B_1}$ form a closed polygonal curve $\Gamma$ in~$R$.
We can assume without loss of generality that $\Gamma$ is simple (see Figure~\ref{fig:redrawing}) and that the $x$-coordinates of $A_0$ and $A_1$ are equal to~$0$.
We also assume that no two vertices of $\Gamma$ except $A_0$ and $A_1$ have the same $x$-coordinates.

We color the points of $\Gamma\cap L_1$ red and the points of $\Gamma\cap P$ blue.
For convenience, we assume that $A_0$ and $A_1$ have both colors.
Let $Z$ denote the interior region delimited by $\Gamma$ including $\Gamma$ itself.
Since $R$ is simply connected, we have $Z\subseteq R$.

Let $x_1<\dotsb<x_n$ be the $x$-coordinates of all vertices of~$\Gamma$.
We use $[n]$ to denote the set of indices $\{1,\dotsc,n\}$.
Since the $x$-coordinates of $A_0$ and $A_1$ are zero, there is $j\in[n]$ such that $x_j=0$.
For $i\in[n]$, we let $l_i$ be the vertical line $\{x_i\}\times\bbR$.
Since the $x$-coordinates of the vertices of $\Gamma\setminus\{A_0,A_1\}$ are distinct, there is at most one vertex of $\Gamma$ on $l_i$ for every $i\in[n]\setminus\{j\}$.
For $i\in[n]$, the intersection of $Z$ with $l_i$ is a family of closed line segments with endpoints from $\Gamma\cap l_i$.
Some of the segments can be \emph{trivial}, that is, consisting of a single point, and some segments can contain a point of $\Gamma$ in their interior.

For $i\in[n]$ and a point $A\in\Gamma\cap l_i$, we say that a point $B$ is a \emph{left neighbor} of $A$ if $B$ lies on $\Gamma\cap l_{i-1}$ and $\overline{AB}\subseteq\Gamma$.
Similarly, $B$ is a \emph{right neighbor} of $A$ if $B\in\Gamma\cap l_{i+1}$ and $\overline{AB}\subseteq\Gamma$.
Note that every point $A\in\Gamma\cap l_i$ has exactly two neighbors and if $A\notin\{A_0,A_1\}$, then the neighbors of $A$ have the same color as~$A$.
We distinguish two types of points of $\Gamma\cap l_i$.
We say that a point $A\in\Gamma\cap l_i$ is \emph{one-sided} if it either has two right or two left neighbors.
Otherwise, we say that $A$ is \emph{two-sided}.
That is, $A$ is two-sided if it has one left and one right neighbor.
See Figure~\ref{fig:polygonal}.

Note that every one-sided point is a vertex of $\Gamma$ and that one-sided points from $\Gamma\cap l_i$ are exactly the points of $\Gamma\cap l_i$ that either form a trivial line segment or that are contained in the interior of some line segment of $Z\cap l_i$.
Consequently, every line segment in $Z\cap l_i$ contains at most one point of $\Gamma$ in its interior.

For $2\leq i\leq n$ and $C,D\in l_i\cap\Gamma$, let $\overline{CD}$ be a line segment in $Z\cap l_i$ whose interior does not contain a point of $\Gamma$ with a left neighbor.
Let $A$ and $B$ be left neighbors of $C$ and $D$, respectively, such that there is no left neighbor of $C$ and $D$ between $A$ and $B$ on~$l_{i-1}$.
Since no point between $A$ and $B$ on $l_{i-1}$ can have a right neighbor, we have $\overline{AB}\subseteq Z\cap l_{i-1}$ and $A,B,C,D$ are vertices of a trapezoid whose interior is contained in~$Z$.
An analogous statement holds for right neighbors of $C$ and $D$ provided that the interior of $\overline{CD}$ does not contain a point of $\Gamma$ with a right neighbor.

\begin{figure}[t]
\centering
\includegraphics{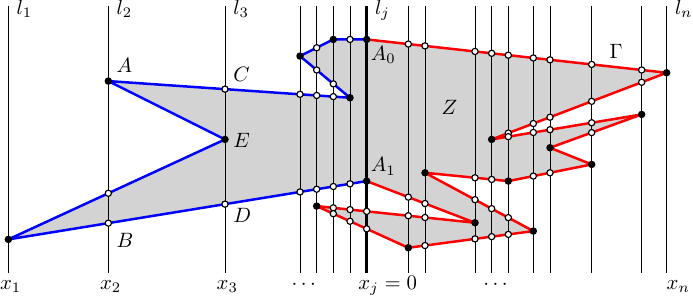}
\caption{Situation in the proof of Lemma \ref{lem:branch-bd-segment-in-S}.
Here $A$ is a left neighbor of $C$ and $B$ is a left neighbor of~$D$.
The points $B$, $C$, and $D$ are two-sided, the points $A$ and $E$ are one-sided.}
\label{fig:polygonal}
\end{figure}

\begin{claim*}\itshape
Let\/ $i\in[n]\setminus\{j\}$, and let\/ $C$ and\/ $D$ be points of\/ $\Gamma\cap l_i$ satisfying\/ $\overline{CD}\subseteq Z\cap l_i$.
Then\/ $C$ and\/ $D$ have the same color.
\end{claim*}

First, we will prove the claim by induction on $i$ for all $i<j$.
The claim clearly holds for $i=1$, as $Z\cap l_1$ contains only a single vertex of~$\Gamma$.
Fix $i$ with $1<i<j$ and suppose that the claim holds for $i-1$.
Let $C,D\in\Gamma\cap l_i$ be points satisfying $\overline{CD}\subseteq Z\cap l_i$.
We show that $C$ and $D$ have the same color.
Obviously, we can assume that the line segment $\overline{CD}$ is non-trivial.
Assume first that the points $C$ and $D$ are two-sided.

Suppose the interior of $\overline{CD}$ does not contain a point of $\Gamma$ with a left neighbor.
Let $A,B\in\Gamma\cap l_{i-1}$ be the left neighbors of $C$ and $D$, respectively.
Then $\overline{AB}\subseteq Z\cap l_{i-1}$.
Thus $A$ and $B$ have the same color by the induction hypothesis.
Since $C,D\notin\{A_0,A_1\}$, the points $A$ and $C$ have the same color as well as the points $B$ and~$D$.
This implies that $C$ and $D$ have the same color too.
If there is a point $E$ of $\Gamma$ in the interior of $\overline{CD}$, then it follows from $R$-convexity of $P$ (Lemma \ref{lem:branch-R-convex}) and $L_1$ (Lemma \ref{lem:L_1-R-convex}) that $E$ has the same color as $C$ and~$D$.

Now, suppose the interior of $\overline{CD}$ contains a point $E$ of $\Gamma$ with a left neighbor.
We have already observed that there is exactly one such point on $\overline{CD}$.
We also know that $E$ has two left neighbors.
The points $C$ and $E$ with their left neighbors $A$ and $B$, respectively, where there is no left neighbor of $E$ between $A$ and $B$ on $l_{i-1}$, form a trapezoid in $Z$ such that $\overline{AB}\subseteq Z\cap l_{i-1}$.
From induction hypothesis $A$ and $B$ have the same color which implies that $C$ and $E$ have the same color as well.
Similarly, $D$ and $E$ have the same color which implies that $C$ and $D$ have the same color as well.

The case where either $C$ or $D$ is one-sided is covered by the previous cases.
The same inductive argument but in the reverse direction shows the claim for all $i$ with $j<i\leq n$.
This completes the proof of the claim.

Now, consider the inclusion-maximal line segment $\overline{CD}$ of $Z\cap l_j$ that contains~$A_0$.
We can assume that either $C$ and $D$ are two-sided.
Suppose for a contradiction that $A_1$ is not contained in $\overline{CD}$.
If $\overline{CD}$ is trivial, that is, $C=D=A_0$, then $A_0$ is one-sided and its neighbors $A$ and $B$ have different colors, as $\Gamma$ changes color in~$A_0$.
This is impossible according to the claim, since we have $\overline{AB}\subseteq Z\cap l_{i-1}$ or $\overline{AB}\subseteq Z\cap l_{i+1}$.
Therefore $\overline{CD}$ is non-trivial.

First, we assume that $A_0$ is an endpoint of $\overline{CD}$, say $C=A_0$.
Then $A_0$ is two-sided.
By symmetry, we can assume that the left neighbor $A$ of $A_0$ and the left neighbor $B$ of $D$ have different colors.
If there is no point of $\Gamma$ with a left neighbor in the interior of $\overline{A_0D}$, then $\overline{AB}\subseteq Z\cap l_{i-1}$.
This is impossible according to the claim.
If there is a point $E\in\Gamma$ with a left neighbor in the interior of $\overline{A_0D}$, then we can use a similar argument either for the line segment $\overline{A_0E}$ or for $\overline{ED}$, as the neighbors of $E$ have the same color.
The last case is when $A_0$ is an interior point of $\overline{CD}$.
Since $\Gamma$ does not change color in $C$ nor in $D$, we apply the claim to one of the line segments $\overline{A_0C}$, $\overline{A_0D}$, and $\overline{CD}$ and show, again, that none of the cases is possible.
Altogether, we have derived a contradiction.

Therefore, $A_0$ and $A_1$ are contained in the same line segment of $Z\cap l_i$.
This completes the proof, as $Z\subseteq R$.
\end{proof}

\begin{lemma}
\label{lem:branch-bd-segment-open}
For every branch\/ $P$ of\/ $R$, the set\/ $P\cap\bd P$ is an open segment.
\end{lemma}

\begin{proof}
Let $P$ be a branch of~$R$.
First, we show that the set $P\cap\bd P$ is convex.
Let $A_0,A_1\in P\cap\bd P$.
By Lemma \ref{lem:branch-bd-segment-in-S}, we have $\overline{A_0A_1}\subseteq R$.
It follows that $\overline{A_0A_1}$ is disjoint from the root of $R$ and thus is contained in~$\inte{R}$.
By compactness, $\overline{A_0A_1}$ has a neighborhood $\nei{\varepsilon}{\overline{A_0A_1}}$ contained in~$\inte{R}$.
Since $P$ is $L_1$-half-open by Lemma \ref{lem:branch-half-open}, there are $B_0,B_1\in\nei{\varepsilon}{\overline{A_0A_1}}\cap L_1$ such that $\overline{A_0B_0}\setminus\{A_0\}\subseteq L_1$ and $\overline{A_1B_1}\setminus\{A_1\}\subseteq L_1$.
For $t\in[0,1]$, let $A_t=(1-t)A_0+t A_1$ and $B_t=(1-t)B_0+tB_1$.
We have $A_t\in\overline{A_0A_1}$ and $B_t\in\overline{B_0B_1}$, hence $A_t,B_t\in\nei{\varepsilon}{\overline{A_0A_1}}$, for all $t\in[0,1]$.
Now, it follows from the $R$-convexity of $P$ (Lemma \ref{lem:branch-R-convex}) and $L_1$ (Lemma \ref{lem:L_1-R-convex}) that $A_t\in P$ and $\overline{A_t B_t}\setminus\{A_t\}\subseteq L_1$, hence $A_t\in P\cap\bd P$, for all $t\in[0,1]$.
This shows that $P\cap\bd P$ is convex.

If $P\cap\bd P$ had three non-collinear points, then they would span a triangle with non-empty interior contained in $P\cap\bd P$, which would be a contradiction.
Since $R$ is bounded, the set $P\cap\bd P$ is a line segment.
That it is an open line segment follows directly from Lemma \ref{lem:branch-half-open}.
\end{proof}

\begin{lemma}
\label{lem:branch-induction}
For every\/ $j\geq 2$, every p-component\/ $P$ of\/ $\bigcup_{k\geq j}L_k$ is a rooted set attached to\/~$L_{j-1}$.
Moreover, for\/ $k\geq 1$, the\/ $k$th level of\/ $P$ is equal to\/ $L_{j-1+k}\cap P$.
\end{lemma}

\begin{proof}
The proof proceeds by induction on~$j$.
For the base case, let $P$ be a p-component of $\bigcup_{i\geq 2}L_i$, that is, a branch of~$R$.
It follows from Lemmas \ref{lem:branch-simply-connected}, \ref{lem:branch-half-open} and \ref{lem:branch-bd-segment-open} that $P$ is a rooted set attached to~$L_1$.
Let $\ell$ be the root of $P$, and let $(L'_k)_{k\geq 1}$ be the partition of $P$ into levels.
We prove that $L'_k\subseteq\bigcup_{i=1}^{k+1}L_i$ and $L_{k+1}\cap P\subseteq\bigcup_{i=1}^kL'_i$ for every $k\geq 1$.

Let $A\in L'_k$.
It follows that there is a polygonal line $\Gamma$ with $k$ line segments connecting $A$ to a point $B\in\ell$.
Moreover, since there is no shorter polygonal line connecting $A$ to $\ell$, the last line segment of $\Gamma$ is not parallel to~$\ell$.
Since $P$ is $L_1$-half-open (Lemma \ref{lem:branch-half-open}), there is a neighborhood $\nei{\varepsilon}{B}$ that is split by $\ell$ into two parts, one of which is a subset of~$L_1$.
Let $C$ be a point in $\nei{\varepsilon}{B}\cap L_1$ such that $\overline{BC}$ is an extension of the last line segment of $\Gamma$.
Since $C\in L_1$, there is a point $D$ on the root of $R$ such that $\overline{CD}\subseteq L_1$.
The polygonal line $\Gamma$ extended by $\overline{BC}$ and $\overline{CD}$ forms a polygonal line with $k+1$ line segments connecting $A$ to the root of~$R$.
This shows that $L'_k\subseteq\bigcup_{i=1}^{k+1}L_i$.

Now, let $A\in L_{k+1}\cap P$.
It follows that there is a polygonal line $\Gamma$ with $k+1$ line segments connecting $A$ to the root of~$R$.
Since $L_1$ is an open subset of $R$ (Lemma~\ref{lem:L_1-half-open}) and $P$ is a p-component of $R\setminus L_1$, there is a point $B\in\Gamma$ such that the part of $\Gamma$ between $A$ and $B$ (inclusive) is contained in $P$ and is maximal with this property.
It follows that $B\in\ell$.
Since $B\notin L_1$, the part of $\Gamma$ between $A$ and $B$ consists of at most $k$ segments.
This shows that $L_{k+1}\cap P\subseteq\bigcup_{i=1}^kL'_i$.

We have thus proved that $L'_k\subseteq\bigcup_{i=1}^{k+1}L_i$ and $L_{k+1}\cap P\subseteq\bigcup_{i=1}^kL'_i$ for every $k\geq 1$.
To conclude the proof of the base case, we note that a straightforward induction shows that $L'_k=L_{k+1}\cap P$ for every $k\geq 1$.

For the induction step, let $j\geq 3$, and let $P$ be a p-component of $\bigcup_{i\geq j}L_i$.
Let $Q$ be the branch of $R$ containing~$P$.
Let $(L'_k)_{k\geq 1}$ be the partition of $Q$ into levels.
As we have proved for the base case, we have $L'_k=L_{k+1}\cap Q$ for every $k\geq 1$.
Hence $P$ is a p-component of $(\bigcup_{i\geq j}L_i)\cap Q=\bigcup_{i\geq j-1}L'_i$.
By the induction hypothesis, $P$ is a rooted set attached to $L'_{j-2}\subseteq L_{j-1}$.
Moreover, for $k\geq 1$, the $k$th level of $P$ is equal to $L'_{j-2+k}\cap P=L_{j-1+k}$.
This completes the induction step and proves the lemma.
\end{proof}

\begin{proof}[Proof of Lemma \ref{lem:branch-structure}]
The statement~\ref{item:body-1} is a direct consequence of the definition of a rooted set, specifically, of the condition that a rooted set is p-connected.

For the proof of the statement~\ref{item:body-2}, let $P$ be a $j$-body of~$R$.
If $j=1$, then $P=L_1=\Vis(r,R)=\Vis(r,L_1)$, where $r$ is the root of $R$, and by Lemmas \ref{lem:L_1-half-open} and \ref{lem:branch-simply-connected}, $L_1$ is rooted with the same root~$r$.
Now, suppose $j\geq 2$.
Let $Q$ be the p-component of $\bigcup_{k\geq j}L_k$ containing~$P$.
By Lemma \ref{lem:branch-induction}, $Q$ is a rooted set and $L_j\cap Q$ is the first level of~$Q$.
Since $P\subseteq L_j\cap Q$, the definition of the first level yields $P=L_j\cap Q=\Vis(r,Q)=\Vis(r,P)$, where $r$ is the root of~$Q$.
By Lemma \ref{lem:L_1-half-open}, $P$ is a rooted set with the same root~$r$.

Lemma \ref{lem:L_1-half-open} and the fact that $L_1$ is p-connected directly yield the statement~\ref{item:body-3}.

Finally, for the proof of the statement~\ref{item:body-4}, let $P$ be a $j$-body of $R$ with $j\geq 2$.
Let $Q$ be the p-component of $\bigcup_{k\geq j}L_k$ containing~$P$.
As we have proved above, $Q$ is a rooted set and $P$ is the first level of $Q$ and shares the root with~$Q$.
Moreover, by Lemma \ref{lem:branch-induction}, $Q$ (and hence $P$) is attached to~$L_{j-1}$.
The definition of attachment implies that $P$ is attached to a single p-component of $L_{j-1}$, that is, a single $(j-1)$-body of~$R$.
\end{proof}

\section{General dimension}
\label{sec:visibility-higher-dim}

This section is devoted to the proofs of Theorem \ref{thm:upper-bound-full-vis} and Theorem \ref{thm:lower-bound}.
In both proofs, we use the operator $\Aff$ to denote the affine hull of a set of points.

\subsection{Proof of Theorem \ref{thm:upper-bound-full-vis}}

Let $T=(B_0,B_1,\dotsc,B_d)$ be a $(d+1)$-tuple of distinct affinely independent points in~$\bbR^d$.
We say that a permutation $B_0,B_1,\dotsc,B_d$ of $T$ is a \emph{regular permutation} of $T$ if the following two conditions hold:
\begin{enumerate}
\item the segment $\overline{B_0B_1}$ is the diameter of $T$,
\item for $i=2,\dotsc,d-1$, the point $B_i$ has the maximum distance to $\Aff(\{B_0,\dotsc,B_{i-1}\})$ among the points $B_i,B_{i+1},\dotsc,B_d$.
\end{enumerate}
Obviously, $T$ has at least two regular permutations due to the interchangeability of $B_0$ and~$B_1$.
The regular permutation $B_{i_0},B_{i_1},\dotsc,B_{i_d}$ with the lexicographically minimal vector $({i_0},{i_1},\dotsc,{i_d})$ is called the \emph{canonical permutation} of~$T$.

Let $T$ be a $(d+1)$-tuple of distinct affinely independent points in $\bbR^d$, and let $B_0,B_1,\dotsc,B_d$ be the canonical permutation of~$T$.
For $i=1,\dotsc,d-1$, we define $\Box_i(T)$ inductively as follows:
\begin{enumerate}
\item $\Box_1(T)\colonequals\overline{B_0B_1}$,
\item for $i=2,\dotsc,d-1$, $\Box_i(T)$ is the box containing all the points $P\in\Aff(\{B_0,B_1,\dotsc,B_i\})$ with the following two properties:
\begin{itemize}
\item the orthogonal projection of $P$ to $\Aff(\{B_0,B_1,\dotsc,B_{i-1}\})$ lies in $\Box_{i-1}(T)$,
\item the distance of $P$ to $\Aff(\{B_0,B_1,\dotsc,B_{i-1}\})$ does not exceed the distance of $B_i$ to $\Aff(\{B_0,B_1,\dotsc,B_{i-1}\})$,
\end{itemize}
\item $\Box_d(T)$ is the box containing all the points $P\in\bbR^d$ such that the orthogonal projection of $P$ to $\Aff(\{B_0,B_1,\dotsc,B_{d-1}\})$ lies in $\Box_{d-1}(T)$ and 
\begin{equation*}
\lambda_d(\Conv(\{B_0,B_1,\dotsc,B_{d-1},P\}))\leq\lambda_d(S)\conv(S).
\end{equation*}
\end{enumerate}
The definition of $\Box_d(T)$ is independent of $B_d$, so we can define
$\Box_d(T\setminus\{B_d\})$ by
\begin{equation*}
\Box_d(T\setminus\{B_d\})\colonequals\Box_d(T).
\end{equation*}
It is not hard to see that this gives us a proper definition of $\Box_d(T^{-})$ for every $d$-tuple $T^{-}$ of $d$ distinct affinely independent points in~$\bbR^d$.

\begin{lemma}
\begin{enumerate}
\item\label{item:boxes-1} For\/ $i=1,\dotsc,d-1$, the box\/ $\Box_i(T)$ contains the orthogonal projection of any point of\/ $T$ to\/ $\Aff(\{B_0,B_1,\dotsc,B_{i-1}\})$.
\item\label{item:boxes-2} If\/ $\Conv(T)\subseteq S$ then\/ $\Box_d(T)$ contains the point\/~$B_d$.
\end{enumerate}
\end{lemma}

\begin{proof}
We prove the statement~\ref{item:boxes-1} by induction on~$i$.
First, let $i=1$.
Then the segment $\Box_1(T)$ must contain every point $A_j\in T$ since otherwise one of the segments $\overline{B_0A_j}$ and $\overline{B_1A_j}$ would be longer than the segment $\overline{B_0B_1}$.
Further, if a point $A_j\in T$ satisfies the statement~\ref{item:boxes-1} for a parameter $i\in\{1,\dotsc,d-2\}$ then it also satisfies the statement~\ref{item:boxes-1} for the parameter $i+1$ since otherwise $A_j\neq B_{i+1}$ and $A_j$ should have been chosen for~$B_{i+1}$.

The statement~\ref{item:boxes-2} follows from the fact that $\Conv(T)\subseteq S$ implies $\lambda_d(\Conv(T))\leq\smc(S)=\lambda_d(S)\conv(S)$.
\end{proof}

For $i=1,\dotsc,d-1$, let $d_i$ be the distance of $B_i$ to $\Aff(\{B_0,B_1,\dotsc,B_{i-1}\})$.
In particular, $d_1$ is the diameter of~$T$.
The following observation follows from the definition of the canonical permutation of $T$ and from the construction of the boxes $\Box_i(T)$.

\begin{observation}
\phantomsection
\label{obs:sides-of-boxes}
\begin{enumerate}
\item\label{item:sides-of-boxes-1} The\/ $(d-1)$-dimensional measure of the simplex\/ $\Conv(T\setminus\{B_d\})$ is equal to\/ $d_1d_2\dotsm d_{d-1}/(d-1)!$.
\item\label{item:sides-of-boxes-2} The sides of\/ $\Box_d(T)$ have lengths\/ $d_1,2d_2,\dotsc,2d_{d-1}$, and\/ $\frac{d!}{d_1d_2\dotsm d_{d-1}}\lambda_d(S)\conv(S)$.
\end{enumerate}
\end{observation}

\begin{proof}[Proof of Theorem \ref{thm:upper-bound-full-vis}]
To estimate $\beer_d(S)$, we partition $\Simp_d(S)$ into the following $d+2$ subsets:
\begin{align*}
X&\colonequals\{T\in\Simp_d(S)\colon\text{$T$ is affinely dependent}\},\\
Y_i&\colonequals\{T=(A_0,\dotsc,A_d)\in\Simp_d(S)\colon\begin{minipage}[t]{210pt}\raggedright
\strut $T$ is affinely independent, and $A_i$ is the last\\
element of the canonical permutation of $T\}$,\\
for $i=0,\dotsc,d$.\strut
\end{minipage}
\end{align*}
We point out that $T$ is considered to be affinely dependent in the above definitions of $X$ and $Y_i$ also in the degenerate case when some point of $S$ appears more than once in~$T$.
We have $\lambda_{d(d+1)}(X)=0$.
Let $i\in\{0,\dotsc,d\}$.
The set $Y_i$ is a subset of the set
\begin{align*}
Y_i'\colonequals\{T=(A_0,\dotsc,A_d)\in S^{d+1}\colon\begin{minipage}[t]{210pt}\raggedright 
\strut $T\setminus\{A_i\}$ is affinely independent and we have\\
$A_i\in\Box_d(T\setminus\{A_i\})\}$.\smash[t]{\strut}
\end{minipage}
\end{align*}
By Observation \ref{obs:sides-of-boxes}~\ref{item:sides-of-boxes-2}, $\lambda_d(\Box_d(T\setminus\{A_i\}))$ is equal to $z\colonequals 2^{d-2}d!\lambda_d(S)\conv(S)$ for every set $T\setminus\{A_i\}$ appearing in the definition of~$Y_i'$.
Therefore, by Fubini's Theorem, the set $Y_i'$ is $\lambda_{d(d+1)}$-measurable and, moreover,
\begin{equation*}
\lambda_{d(d+1)}(Y_i')=(\lambda_d(S))^dz=(\lambda_d(S))^{d+1}2^{d-2}d!\conv(S).
\end{equation*}
Thus,
\begin{equation*}
\beer_d(S)=\frac{\lambda_{d(d+1)}(\Simp_d(S))}{\lambda_d(S)^{d+1}}\leq\frac{\lambda_{d(d+1)}(X)+\sum_{i=0}^d\lambda_{d(d+1)}(Y_i')}{\lambda_d(S)^{d+1}}=2^{d-2}(d+1)!\conv(S).
\end{equation*}
This completes the proof of Theorem \ref{thm:upper-bound-full-vis}.
\end{proof}

\subsection{Proof of Theorem \ref{thm:lower-bound}}

In the following, we make no serious effort to optimize the constants.
As the first step towards the proof of Theorem \ref{thm:lower-bound}, we show that if we remove an arbitrary $n$-tuple of points from the open $d$-dimensional box $(0,1)^d$, then the $d$-index of convexity of the resulting set is of order $\Omega(\frac{1}{n})$.

\begin{lemma}
\label{lem:holes}
For every positive integer\/ $n$ and every\/ $n$-tuple\/ $N$ of points from\/ $(0,1)^d$, the set\/ $S\colonequals(0,1)^d\setminus N$ satisfies\/ $\beer_d(S)\geq 1/2n$.
\end{lemma}

\begin{proof}
Let $S$ and $N=\{B_1,\dotsc,B_n\}$ be the sets from the statement and let $0$ be the origin.
We use $S^{d-1}_*$ to denote the set of $(d-1)$-tuples $(A_1,\dotsc,A_{d-1})\in S^{d-1}$ that satisfy the following: for every $B\in N$ the points $A_1,\dotsc,A_{d-1},B$ are affinely independent and $\Aff(\{A_1,\dotsc,A_{d-1},B\})\cap(N\cup\{0\})=\{B\}$.
Note that the set $S^{d-1}_*$ is measurable and $\lambda_{d(d-1)}(S^{d-1}_*)=1$.
If $h$ is a hyperplane in $\bbR^d$ that does not contain the origin, we use $h^-$ and $h^+$ to denote the open half-spaces defined by $h$ such that $0\in h^-$.

Let $(A_1,\dotsc,A_{d-1})\in S^{d-1}_*$.
For a point $B\in N$, we let $h_{A_1,\dotsc,A_{d-1},B}$ be the hyperplane determined by the $d$-tuple $(A_1,\dotsc,A_{d-1},B)$.
Since $(A_1,\dotsc,A_{d-1})\in S^{d-1}_*$, we see that $h_{A_1,\dotsc,A_{d-1},B}$ satisfies $h_{A_1,\dotsc,A_{d-1},B}\cap N=\{B\}$ and that it does not contain the origin.
Therefore the half-spaces $\smash[t]{h_{A_1,\dotsc,A_{d-1},B}^-}$ and $\smash[t]{h_{A_1,\dotsc,A_{d-1},B}^+}$ are well defined.

For every $(d-1)$-tuple $(A_1,\dotsc,A_{d-1})\in S^{d-1}_*$, we split the set $S$ into $2n$ pairwise disjoint open convex sets that are determined by the hyperplanes $h_{A_1,\dotsc,A_{d-1},B}$ for $B\in N$.
This is done by induction on~$n$.
For $n=1$, we set $P_1(A_1,\dotsc,A_{d-1})\colonequals S\cap\smash[t]{h_{A_1,\dotsc,A_{d-1},B_1}^-}$ and $P_2(A_1,\dotsc,A_{d-1})\colonequals S\cap\smash[t]{h_{A_1,\dotsc,A_{d-1},B_1}^+}$.
Suppose we have split the set $S$ into sets $P_i(A_1,\dotsc,A_{d-1})$ for $1\leq i\leq 2(n-1)$ and $n\geq 2$.
Consider the hyperplane $h_{A_1,\dotsc,A_{d-1},B_n}$.
Since for every $k\in\{1,\dotsc,n-1\}$ the intersection $h_{A_1,\dotsc,A_{d-1},B_k}\cap h_{A_1,\dotsc,A_{d-1},B_n}$ is the affine hull of $A_1,\dotsc,A_{d-1}$, we see that $h_{A_1,\dotsc,A_{d-1},B_n}\setminus\Aff(\{A_1,\dotsc,A_{d-1}\})$ is contained in two sets $P_i(A_1,\dotsc,A_{d-1})$ and $P_j(A_1,\dotsc,A_{d-1})$ for some $1\leq i<j\leq 2(n-1)$.
We restrict these sets to their intersection with the half-space $\smash[t]{h_{A_1,\dotsc,A_{d-1},B_n}^-}$ and set $P_{2n-1}(A_1,\dotsc,A_{d-1})$ and $P_{2n}(A_1,\dotsc,A_{d-1})$ as the intersection of $\smash[t]{h_{A_1,\dotsc,A_{d-1},B_n}^+}$ with $P_i(A_1,\dotsc,A_{d-1})$ and $P_j(A_1,\dotsc,A_{d-1})$, respectively.
See Figure~\ref{fig:strips} for an illustration.

\begin{figure}[t]
\centering
\includegraphics{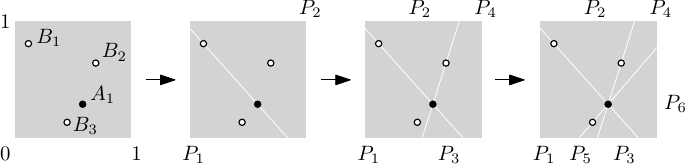}
\caption{The inductive splitting of $S=(0,1)^2\setminus N$ with respect to the point~$A_1$.
The points $B_1,B_2,B_3$ of $N$ are denoted as empty circles and we use a shorthand $P_i$ for $P_i(A_1)$.}
\label{fig:strips}
\end{figure}

Since none of the sets $P_i(A_1,\dotsc,A_{d-1})$ contains a point from $N$, it can be regarded as an intersection of $(0,1)^d$ with open half-spaces.
Therefore every set $P_i(A_1,\dotsc,A_{d-1})$ is an open convex subset of~$S$.
Let $P(A_1,\dotsc,A_{d-1})$ be the set $S\setminus\bigl(\bigcup_{B\in N}h_{A_1,\dotsc,A_{d-1},B}\bigr)$.
Clearly, $\lambda_d(P(A_1,\dotsc,A_{d-1}))=1$.
Since the sets $P_i(A_1,\dotsc,A_{d-1})$ form a partitioning of $P(A_1,\dotsc,A_{d-1})$, we also have $\sum_{i=1}^{2n}\lambda_d(P_i(A_1,\dotsc,A_{d-1}))=1$.

For $i=1,\dotsc,2n$, we let $\calR_i$ be the subset of $S^{d-1}_*\times S^2$ defined as
\begin{equation*}
\calR_i\colonequals\{(A_1,\dotsc,A_{d+1})\in S^{d-1}_*\times S^2\colon A_d,A_{d+1}\in P_i(A_1,\dotsc,A_{d-1})\},
\end{equation*}
and we let $\calR\colonequals\bigcup_{i=1}^{2n}\calR_i$.
The sets $\calR_i$ are pairwise disjoint and it is not difficult to argue that these sets are measurable.
If a $(d+1)$-tuple $(A_1,\dotsc,A_{d+1})$ is contained in $\calR_i$ for some $i\in\{1,\dotsc,2n\}$, then $(A_1,\dotsc,A_{d+1})$ is contained in $\Simp_d(S)$, as $P_i(A_1,\dotsc,A_{d-1})\cup(\Aff(\{A_1,\dotsc,A_{d-1}\})\cap S)$ is a convex subset of~$S$.
Therefore, to find a lower bound for $\beer_d(S)=\lambda_{d(d+1)}(\Simp_d(S))$, it suffices to find a lower bound for $\lambda_{d(d+1)}(\calR)$, because $\calR$ is a subset of $\Simp_d(S)$.
By Fubini's Theorem, we obtain
\begin{align*}
\lambda_{d(d+1)}(\calR_i)&=\int_{(A_1,\dotsc,A_{d+1})\in S^{d-1}_*\times S^2}[(A_1,\dotsc,A_{d+1})\in\calR_i]\\
&=\int_{(A_1,\dotsc,A_{d-1})\in S^{d-1}_*}\left(\int_{(A_d,A_{d+1})\in S^2}[A_d,A_{d+1}\in P_i(A_1,\dotsc,A_{d-1})]\right)\\
&=\int_{(A_1,\dotsc,A_{d-1})\in S^{d-1}_*}\lambda_d(P_i(A_1,\dotsc,A_{d-1}))^2,
\end{align*}
where $[\phi]$ is the characteristic function of a logical expression $\phi$, that is, $[\phi]$ equals $1$ if the condition $\phi$ holds and $0$ otherwise.
For the measure of $\calR$ we then derive
\begin{equation*}
\lambda_{d(d+1)}(\calR)=\sum_{i=1}^{2n}\lambda_{d(d+1)}(\calR_i)=\int_{(A_1,\dotsc,A_{d-1})\in S^{d-1}_*}\sum_{i=1}^{2n}\left(\lambda_d(P_i(A_1,\dotsc,A_{d-1})\right)^2.
\end{equation*}
Since the function $x\mapsto x^2$ is convex, we can apply Jensen's inequality and bound the last term from below by
\begin{equation*}
\int_{(A_1,\dotsc,A_{d-1})\in S^{d-1}_*} 2n\left(\frac{\sum_{i=1}^{2n}\lambda_d(P_i(A_1,\dotsc,A_{d-1}))}{2n}\right)^2=\frac{1}{2n}.\qedhere
\end{equation*}
\end{proof}

The next step in the proof of Theorem \ref{thm:lower-bound} is to find a convenient $n$-tuple $N$ of points from $(0,1)^d$ whose removal produces a set with sufficiently small convexity ratio.
We are going to find $N$ using a continuous version of the well-known Epsilon Net Theorem \cite{HaW87}.
Before stating this result, we need some definitions.

Let $X$ be a subset of $\bbR^d$ and let $\calU$ be a set system on~$X$.
We say that a set $T\subseteq X$ is \emph{shattered} by $\calU$ if every subset of $T$ can be obtained as the intersection of some $U\in\calU$ with~$T$.
The \emph{Vapnik-Chervonenkis dimension} (or \emph{VC-dimension}) of $\calU$, denoted by $\dim(\calU)$, is the maximum $n$ (or $\infty$ if no such maximum exists) for which some subset of $X$ of cardinality $n$ is shattered by~$\calU$.

Let $\calU$ be a system of measurable subsets of a set $X\subseteq\bbR^d$ with $\lambda_d(X)=1$, and let $\varepsilon\in(0,1)$ be a real number.
A set $N\subseteq X$ is called an \emph{$\varepsilon$-net} for $(X,\calU)$ if $N\cap U\neq\emptyset$ for every $U\in\calU$ with $\lambda_d(U)\geq\varepsilon$.

\begin{theorem}[{\cite[Theorem 10.2.4]{Mat02}}]
\label{thm:epsilon-net}
Let\/ $X$ be a subset of\/ $\bbR^d$ with\/ $\lambda_d(X)=1$.
Then for every system\/ $\calU$ of measurable subsets of\/ $X$ with\/ $\dim(\calU)\leq v$, $v\geq 2$, there is a\/ $\frac{1}{r}$-net for\/ $(X,\calU)$ of size at most\/ $2vr\log_2 r$ for\/ $r$ sufficiently large with respect to\/~$v$.
\end{theorem}

To apply Theorem \ref{thm:epsilon-net}, the VC-dimension of the set system $\calU$ has to be finite.
However, it is known that the VC-dimension of all convex sets in $\bbR^d$ is infinite (see e.g.\ \cite[page 238]{Mat02}).
Therefore, instead of considering convex sets directly, we approximate them by ellipsoids.

A \emph{$d$-dimensional ellipsoid} in $\bbR^d$ is an image of the closed $d$-dimensional unit ball under a nonsingular affine map.
A \emph{convex body} in $\bbR^d$ is a compact convex subset of $\bbR^d$ with non-empty interior.
The following result, known as John's Lemma \cite{Joh48}, shows that every convex body can be approximated by an inscribed ellipsoid.

\begin{lemma}[{\cite[Theorem 13.4.1]{Mat02}}]
\label{lem:john-ellipsoid}
For every\/ $d$-dimensional convex body\/ $K\subseteq\bbR^d$, there is a\/ $d$-dimensional ellipsoid\/ $E$ with the center\/ $C$ that satisfies
\begin{equation*}
E\subseteq K\subseteq C+d(E-C).
\end{equation*}
In particular, we have\/ $\lambda_d(K)/d^d\leq\lambda_d(E)$.
\end{lemma}

As the last step before the proof of Theorem \ref{thm:lower-bound}, we mention the following fact, which implies that the VC-dimension of the system $\calE$ of $d$-dimensional ellipsoids in $\bbR^d$ is at most $\binom{d+2}d$.

\begin{lemma}[{\cite[Proposition 10.3.2]{Mat02}}]
\label{lem:bounded-VC-dim}
Let\/ $\bbR[x_1,\dotsc,x_d]_{\leq t}$ denote the set of real polynomials in\/ $d$ variables of degree at most\/ $t$, and let
\begin{equation*}
\calP_{d,t}=\bigl\{\{x\in\bbR^d\colon p(x)\geq 0\}\colon p\in\bbR[x_1,\dotsc,x_d]_{\leq t}\bigr\}.
\end{equation*}
Then\/ $\dim(\calP_{d,t})\leq\binom{d+t}{d}$.
\end{lemma}

\begin{proof}[Proof of Theorem \ref{thm:lower-bound}]
Suppose we are given $\varepsilon>0$ which is sufficiently small with respect to~$d$.
We show how to construct a set $S\subseteq\bbR^d$ with $\lambda_d(S)=1$ satisfying $\conv(S)\leq\varepsilon$ and
\begin{equation*}
\beer_d(S)\geq\frac{1}{8\binom{d+2}{d}d^d}\cdot\frac{\varepsilon}{\log_2{1/\varepsilon}}.
\end{equation*}
Without loss of generality we assume that $\varepsilon=d^d/r$ for some integer $r\geq 2d^{2d}$.

Consider the open $d$-dimensional box $(0,1)^d$ and the system $\calE\restriction(0,1)^d$ of $d$-dimensional ellipsoids in $(0,1)^d$.
Since the restriction of $\calE$ to $(0,1)^d$ does not increase the VC-dimension, Lemma \ref{lem:bounded-VC-dim} implies $\dim(\calE\restriction(0,1)^d)\leq\binom{d+2}{d}$.

If we set $n\colonequals 2\binom{d+2}{d}r\lceil\log_2{r}\rceil$, then, by Theorem \ref{thm:epsilon-net}, there is a $\frac{1}{r}$-net $N$ for the system $((0,1)^d,\calE\restriction(0,1)^d)$ of size $n$, having $r$ sufficiently large with respect to~$d$.
Let $S$ be the set $(0,1)^d\setminus N$.
Clearly, we have $\lambda_d(S)=1$.

Suppose $K$ is a convex subset of $(0,1)^d$ with $\lambda_d(K)>\varepsilon$.
Since the measure of $K$ is positive, we can assume that $K$ is a convex body of measure at least~$\varepsilon$.
By Lemma~\ref{lem:john-ellipsoid}, the convex body $K$ contains a $d$-dimensional ellipsoid $E$ with $\lambda_d(E)\geq\varepsilon/d^d=\frac{1}{r}$.
Therefore $E\cap N\neq\emptyset$.
Since we have $E\subseteq K$ and $N\cap S=\emptyset$, we see that $K$ is not a subset of~$S$.
In other words, we have $\conv(S)\leq\varepsilon$.

By Lemma \ref{lem:holes}, we have $\beer_d(S)\geq\frac{1}{2n}$.
According to the choice of $n$ and $r$, the term $\frac{1}{2n}$ is bounded from below by
\begin{equation*}
\frac{1}{4\binom{d+2}{d}r\log_2{2r}}=\frac{\varepsilon}{4\binom{d+2}{d}d^d\log_2{(2d^d/\varepsilon)}}\geq\frac{1}{8\binom{d+2}{d}d^d}\cdot\frac{\varepsilon}{\log_2{1/\varepsilon}},
\end{equation*}
where the last inequality follows from the estimate $2d^d\leq 1/\varepsilon$.
This completes the proof of Theorem \ref{thm:lower-bound}.
\end{proof}

It is a natural question whether the bound for $\beer_d(S)$ in Theorem \ref{thm:lower-bound} can be improved to $\beer_d(S)=\Omega(\conv(S))$.
In the plane, this is related to the famous problem of Danzer and Rogers (see \cite{BrC97,PaT12} and Problem E14 in \cite{CFG91}) which asks whether for given $\varepsilon>0$ there is a set $N'\subseteq(0,1)^2$ of size $O(\frac{1}{\varepsilon})$ with the property that every convex set of area $\varepsilon$ within the unit square contains at least one point from~$N'$.

If this problem was to be answered affirmatively, then we could use such a set $N'$ to stab $(0,1)^2$ in our proof of Theorem \ref{thm:lower-bound} which would yield the desired bound for $\beer_2(S)$.
However it is generally believed that the answer is likely to be nonlinear in $\frac{1}{\varepsilon}$.

\subsection{A set with large \texorpdfstring{$k$}{k}-index of convexity and small convexity ratio}

\begin{proposition}
\label{prop:example}
For every integer\/ $d\geq 2$, the set\/ $S\colonequals[0,1]^d\setminus\bbQ^d$ satisfies\/ $\conv(S)=0$ and\/ $\beer_k(S)=1$ for every positive integer\/ $k<d$.
\end{proposition}

\begin{proof}
Since $\bbQ^d$ is countable and $\lambda_d([0,1]^d)=1$, we have $\lambda_d(S)=1$.
Every convex subset $K$ of $[0,1]^d$ with positive $d$-dimensional measure contains an open $d$-dimensional ball $\fB$ with positive diameter, as there is a $(d+1)$-tuple of affinely independent points of~$K$.
Since $\bbQ^d$ is a dense subset of $\bbR^d$, we see that $\fB\cap\bbQ^d\neq\emptyset$ and thus $\conv(S)=0$.

It remains to estimate $\beer_k(S)$.
By Fubini's Theorem, we have
\begin{equation*}
\beer_k(S)=\int_{(B_1,\dotsc,B_k)\in S^k}\frac{\lambda_d\bigl(\{A\in S\colon\Conv(\{B_1,\dotsc,B_k,A\})\subseteq S\}\bigr)}{\lambda_d(S)^{k+1}}.
\end{equation*}

If $A$ is a point of $S$ such that $\Conv(\{B_1,\dotsc,B_k,A\})$ is not contained in $S$, then $A$ is a point of the affine hull $\Aff(\{B_1,\dotsc,B_k,Q\})$ of $B_1,\dotsc,B_k$ and some $Q\in\bbQ^d$.
Therefore, $\beer_k(S)$ is at least
\begin{equation*}
\int_{(B_1,\dotsc,B_k)\in S^k}\frac{\lambda_d([0,1]^d)-\lambda_d\bigl(\bigcup_{Q\in\bbQ^d}\Aff(\{B_1,\dotsc,B_k,Q\})\bigr)}{\lambda_d(S)^{k+1}}.
\end{equation*}
A countable union of affine subspaces of dimension less than $d$ has $d$-dimensional measure zero and we already know that $\lambda_d(S)=1=\lambda_d([0,1]^d)$, hence $\beer_k(S)=1$.
\end{proof}

\section{Other variants and open problems}
\label{sec:generalizations}

We have seen in Theorem \ref{thm:upper-bound-2d} that a p-componentwise simply connected set $S\subseteq\bbR^2$ whose $\beer(S)$ is defined satisfies $\beer(S)\leq\alpha\conv(S)$, for an absolute constant $\alpha\leq 180$.
Equivalently, such a set $S$ satisfies $\smc(S)\geq\beer(S)\lambda_2(S)/180$.

By a result of Blaschke \cite{Bla17} (see also Sas \cite{Sas39}), every convex set $K\subseteq\bbR^2$ contains a triangle of measure at least $\frac{3\sqrt{3}}{4\pi}\lambda_2(K)$.
In view of this, Theorem \ref{thm:upper-bound-2d} yields the following consequence.

\begin{corollary}
\label{cor:upper-bound-2d-triangle}
There is a constant\/ $\alpha>0$ such that every p-componentwise simply connected set\/ $S\subseteq\bbR^2$ whose\/ $\beer(S)$ is defined contains a triangle\/ $T\subseteq S$ of measure at least\/ $\alpha\beer(S)\lambda_2(S)$.
\end{corollary}

A similar argument works in higher dimensions as well.
For every $d\geq 2$, there is a constant $\beta=\beta(d)$ such that every convex set $K\subseteq\bbR^d$ contains a simplex of measure at least $\beta\lambda_d(K)$ (see e.g.\ Lassak \cite{Las11}).
Therefore, Theorem \ref{thm:upper-bound-full-vis} can be rephrased in the following equivalent form.

\begin{corollary}
\label{cor:upper-bound-full-vis-simplex}
For every\/ $d\geq 2$, there is a constant\/ $\alpha=\alpha(d)>0$ such that every set\/ $S\subseteq\bbR^d$ whose\/ $\beer_d(S)$ is defined contains a simplex\/ $T$ of measure at least\/ $\alpha\beer_d(S)\lambda_d(S)$.
\end{corollary}

What can we say about sets $S\subseteq\bbR^2$ that are not p-componentwise simply connected?
First of all, we can consider a weaker form of simple connectivity: we call a set $S$ \emph{p-componentwise simply\/ $\triangle$-connected} if for every triangle $T$ such that $\bd T\subseteq S$ we have $T\subseteq S$.
We conjecture that Theorem \ref{thm:upper-bound-2d} can be extended to p-componentwise simply $\triangle$-connected sets.

\begin{conjecture}
There is an absolute constant\/ $\alpha>0$ such that every p-componentwise simply\/ $\triangle$-connected set\/ $S\subseteq\bbR^2$ whose\/ $\beer(S)$ is defined satisfies\/ $\beer(S)\leq\alpha\conv(S)$.
\end{conjecture}

What does the value of $\beer(S)$ say about a planar set $S$ that does not satisfy even a weak form of simple connectivity?
As Proposition \ref{prop:example} shows, such a set may not contain any convex subset of positive measure, even when $\beer(S)$ is equal to~$1$.
However, we conjecture that a large $\beer(S)$ implies the existence of a large convex set whose boundary belongs to~$S$.

\begin{conjecture}
\label{conj:boundary}
For every\/ $\varepsilon>0$, there is a\/ $\delta>0$ such that if\/ $S\subseteq\bbR^2$ is a set with\/ $\beer(S)\geq\varepsilon$, then there is a bounded convex set\/ $C\subseteq\bbR^2$ with\/ $\lambda(C)\geq\delta\lambda(S)$ and\/ $\bd C\subseteq S$.
\end{conjecture}

Theorem \ref{thm:upper-bound-2d} shows that Conjecture \ref{conj:boundary} holds for p-componentwise simply connected sets, with $\delta$ being a constant multiple of~$\varepsilon$.
It is possible that even in the general setting of Conjecture \ref{conj:boundary}, $\delta$ can be taken as a constant multiple of~$\varepsilon$.

Motivated by Corollary \ref{cor:upper-bound-2d-triangle}, we propose a stronger version of Conjecture \ref{conj:boundary}, where the convex set $C$ is required to be a triangle.

\begin{conjecture}
\label{conj:boundary-simplex}
For every\/ $\varepsilon>0$, there is a\/ $\delta>0$ such that if\/ $S\subseteq\bbR^2$ is a set with\/ $\beer(S)\geq\varepsilon$, then there is a triangle\/ $T\subseteq\bbR^2$ with\/ $\lambda(T)\geq\delta\lambda(S)$ and\/ $\bd T\subseteq S$.
\end{conjecture}

Note that Conjecture \ref{conj:boundary-simplex} holds when restricted to p-componentwise simply connected sets, as implied by Corollary \ref{cor:upper-bound-2d-triangle}.

We can generalize Conjecture \ref{conj:boundary-simplex} to higher dimensions and to higher-order indices of convexity.
To state the general conjecture, we introduce the following notation: for a set $X\subseteq\bbR^d$, let $\binom{X}{k}$ be the set of $k$-element subsets of $X$, and let the set $\Skel_k(X)$ be defined by
\begin{equation*}
\Skel_k(X)\colonequals\bigcup_{Y\in\binom{X}{k+1}}\Conv(Y).
\end{equation*}
If $X$ is the vertex set of a $d$-dimensional simplex $T=\Conv(X)$, then $\Skel_k(X)$ is often called the \emph{$k$-dimensional skeleton} of~$T$.
Our general conjecture states, roughly speaking, that sets with large $k$-index of convexity should contain the $k$-dimensional skeleton of a large simplex.
Here is the precise statement.

\begin{conjecture}
For every\/ $k,d\in\bbN$ such that\/ $1\leq k\leq d$ and every\/ $\varepsilon>0$, there is a\/ $\delta>0$ such that if\/ $S\subseteq\bbR^d$ is a set with\/ $\beer_k(S)\geq\varepsilon$, then there is a simplex\/ $T$ with vertex set\/ $X$ such that\/ $\lambda_d(T)\geq\delta\lambda_d(S)$ and\/ $\Skel_k(X)\subseteq S$.
\end{conjecture}

Corollary \ref{cor:upper-bound-full-vis-simplex} asserts that this conjecture holds in the special case of $k=d\geq 2$, since $\Skel_d(X)=\Conv(X)=T$.
Corollary \ref{cor:upper-bound-2d-triangle} shows that the conjecture holds for $k=1$ and $d=2$ if $S$ is further assumed to be p-componentwise simply connected.
In all these cases, $\delta$ can be taken as a constant multiple of $\varepsilon$, with the constant depending on $k$ and~$d$.

Finally, we can ask whether there is a way to generalize Theorem \ref{thm:upper-bound-2d} to higher dimensions, by replacing simple connectivity with another topological property.
Here is an example of one such possible generalization.

\begin{conjecture}
\label{conj:contractible}
For every\/ $d\geq 2$, there is a constant\/ $\alpha=\alpha(d)>0$ such that if\/ $S\subseteq\bbR^d$ is a set with\/ $\beer_{d-1}(S)$ defined whose every p-component is contractible, then\/ $\beer_{d-1}(S)\leq\alpha\conv(S)$.
\end{conjecture}

A modification of the proof of Theorem \ref{thm:upper-bound-full-vis} implies that Conjecture \ref{conj:contractible} is true for star-shaped sets~$S$.

\section*{Acknowledgment}

The authors would like to thank Marek Eliáš for interesting discussions about the problem and participation in our meetings during the early stages of the research.

\end{document}